\documentclass[11pt,final]{amsart}
\usepackage{amsmath,latexsym,amssymb,amsmath,%awheremscd,
amscd,amsthm,amsxtra,color}
\usepackage{amsfonts}
\usepackage{amssymb,latexsym}
\usepackage{enumerate}
\usepackage{mathrsfs}
\usepackage{amsmath}
\usepackage{amsthm}
\usepackage{verbatim}

\newtheorem{thm}{Theorem}[section]
\newtheorem{prop}[thm]{Proposition}
\newtheorem{lem}[thm]{Lemma}
\newtheorem{cor}[thm]{Corollary}
\newtheorem{obs}{Observation}[section]
\newtheorem{clm}{Claim}[section]

\newtheorem{rem}{\textbf{Remark}}[section]
\def\nz{\mathbb{N}}

\def\rz{\mathbb{R}}

\def\O{\Omega}
\def\p{\partial}
\def\e{\epsilon}
\def\a{\alpha}
\def\b{\beta}
\def\g{\gamma}
\def\G{\Gamma}
\def\d{\delta}
\def\l{\lambda}
\def\L{\Lambda}
\def\s{\sigma}
\def\vp{\varphi}

\def\ra{\rightarrow}
\def\ru{\rightharpoonup}
\def\ov{\overline}

\def\RA{\Rightarrow}
\def\bs{\backslash}
\def\tl{\tilde}

\def\n{\nabla}
\def\H{{\mathcal H}}
\def\non{\nonumber}
\def\div{{\rm div}}

\newcommand{\D}{\displaystyle}
\newcommand{\T}{\text}

\begin{document}

\title{Classification Theorem For Positive Critical Points Of Sobolev Trace Inequality}
\author{Yang Zhou}
\date{}
\maketitle
\noindent{\bf Abstract:}
 We consider the Euler-Lagrange equation of Sobolev trace inequality and prove several classification results. Exploiting the moving sphere method, it has been shown, when $p=2$, positive solutions of Euler-Lagrange equation of Sobolev trace inequality are classified.
 Since the moving sphere method strongly relies on the symmetries of the equation, in this paper we use asymptotic estimates and two important integral identities to classify positive solutions of Euler-Langrange equation of Sobolev trace inequality under finite energy when $1<p<n$.

\section{Introduction}
Given $n\ge 2$ and $1<p<n$, firstly we consider the following equation in half space
\begin{equation}\label{e1}
    \left\{
    \begin{array}{lr}
	\Delta_p u(y,t)=0 & (y,t)\in \mathbb{R}^n_+=\mathbb{R}^{n-1}\times\{t>0\}\\
	|D u|^{p-2}\frac{\partial u}{\partial t}(y,t)=-u^q & t=0\\
    u>0 & \T{in}\,\,\, \mathbb{R}^n_+\\
     u\in D^{1,p}(\mathbb{R}^n_+)
	\end{array}
 \right.     
\end{equation}
where $$D^{1,p}(\mathbb{R}^n_+)=\{u\in L^{p^*}(\mathbb{R}^n_+):\nabla u\in L^p(\mathbb{R}^n_+)\},\,q=\frac{n(p-1)}{n-p},\,p^*=\frac{n p}{n-p}.
$$
We can see $(\ref{e1})$ is the Euler-Lagrange equation of the following Sobolev trace inequality: 
\begin{equation}\label{e2}
\|u\|_{L^{p_*}(\p\mathbb{R}^n_+)}\le C_{n,p}\|\n u\|_{L^p(\mathbb{R}^n_+)}
\end{equation}
where $p_*=\frac{(n-1)p}{n-p}$, $u\in C_c^{\infty}(\ov{\mathbb{R}^n_+})$.

The research for optimal constants in $(\ref{e2})$ and the classification of positive solutions to
$(\ref{e1})$
started in the seminal paper \cite{E88,B93} and it has been the object of
several studies. J. F. Escobar \cite{E88} conjectured that the minimizers for the Sobolev quotient,
$$
Q_p(\mathbb{R}^n_+):=\inf_\vp\left\{\frac{\|\n\vp\|_{L^p(\mathbb{R}^n_+)}^p}{\|\vp\|_{L^{p_*}(\p\mathbb{R}^n_+)}^p}:\vp\in C_0^\infty(\ov{\mathbb{R}^n_+})\right\},
$$
are the functions
\begin{equation}
   \label{e2.5}u(y,t)=\bigg(\frac{\l^{\frac{2}{p}}}{|y-y_0|^2+(t+\l)^2}\bigg)^{\frac{n-p}{2(p-1)}} 
\end{equation}
with $\l>0$ and $y_0\in\rz^{n-1}$, and he solved the conjecture for $p=2$ (see also \cite{B93}). Also, he raised a question in \cite{E88}: Does every positive solution to the Euler-Lagrange $(\ref{e1})$ have the form like $(\ref{e2.5})$ when $p=2$?

 By means of mass transportation techniques, B. Nazaret \cite{N06} obtained the optimal constants in Sobolev trace inequality for every $p\in(1, n)$, along with the optimality of $(\ref{e2.5})$. However, the characterization of equality cases is
still missing in \cite{N06}, but this only depends on some minor technical points that were filled in \cite{MN17}.

For $p=2$, equation $(\ref{e1})$ becomes the following equation which is related to the critical fractional Laplacian equation (see \cite{CS07})
\begin{equation*}
    \left\{
    \begin{array}{lr}
	\Delta u(y,t)=0 & (y,t)\in \mathbb{R}^n_+\\
	\frac{\partial u}{\partial t}(y,t)=-u^\frac{n}{n-2} & t=0.
	\end{array}
 \right.     
\end{equation*}
It was shown in \cite{HY94,O96,YYL03} using the moving sphere method that positive solutions to the above equation must be of the form
$$
u(y,t)=\bigg(\frac{\l(n-2)}{|y-y_0|^2+(t+\l)^2}\bigg)^{\frac{n-2}{2}}
$$
with $\l>0$ and $y_0\in\rz^{n-1}$.

Note that for $p\neq 2$, the problem is quasilinear and the Kelvin transform is not
available, which makes it more complicated then the semilinear case. Recall that for the whole space,
under the additional assumption of finite energy, B. Sciunzi \cite{S16} and J. V\'{e}tois \cite{Veto16}
exploited the moving plane method to show that any
positive weak solution to $\Delta_p u+u^{p^*-1}\,\,\T{in}\,\,\rz^n$ must be of the form 
$$
u(x)=\bigg(\frac{|x-x_0|^{\frac{p}{p-1}}+\l^{\frac{p}{p-1}}}{\l^{\frac{1}{p-1}}n^{\frac{1}{p}}(\frac{n-p}{p-1})^{\frac{p-1}{p}}}\bigg)^{-\frac{n-p}{p}}
$$
with $\l>0$ and $x_0\in\rz^{n}$. Later, based on the asymptotic estimates in \cite{Veto16}, the paper \cite{CFR20} used integration method to classify positive solutions to the critical p-Laplacian equation in an anisotropic setting, which could also imply this result.
In this paper, we adapt the proof of \cite{CFR20} and classify the positive
solutions of $(\ref{e1})$, which answers Escobar's question for general $p\in(1,n)$. 

\begin{thm}\label{t1}
Let $n\ge 2,\,1<p<n,$ and let $u$ be a solution to $(\ref{e1})$. Then 
\begin{equation}
u(y,t)=\bigg(\frac{n-p}{p-1}\bigg)^{\frac{n-p}{p}}\bigg(\frac{\l^{\frac{2}{p}}}{|y-y_0|^2+(t+\l)^2}\bigg)^{\frac{n-p}{2(p-1)}}
 \label{e1.1}
\end{equation}
for some $\l>0$ and $y_0\in \rz^{n-1}$.
\end{thm}

\begin{rem}\label{r1.1}
   We say $u$ is a solution to $(\ref{e1})$, if 
   \begin{equation}
   \int_\O |\n u|^{p-2}\n u\cdot\n \vp d x=\int_{\p\O}u^q \vp d\s,\,\,\,\forall \vp\in C_c^\infty(\rz^n)
    \label{er1}  
   \end{equation}
\end{rem}

The proof of Theorem \ref{t1} is based on asymptotic estimates of $u$ and several integral identities. On the one hand, note that for the asymptotic upper bound of solutions to $(\ref{e1})$, there is much research on it (see e.g. \cite{W10,MW19}). Following the proof in \cite{W10,MW19} and arguing as \cite{CFR20,L19}, we can show that $u$ is bounded in $\mathbb{R}^n_+$ and $u$ behaves as
the fundamental solution from above. However, we notice that, differently from \cite{CFR20,Veto16}, we have to prove a Harnack type inequality for equation $(\ref{e1})$ in order to get asymptotic lower bounds on $u$. Adapting the proof in \cite{HL00}, we obtained a Harnack type inequality under small energy (see Corollary \ref{cor1}). Then arguing as \cite{Veto16}, we obtain asymptotic lower bounds on $u$. On the other hand, unlike the whole space case, we need not only Serrin-Zou identity (see e.g. \cite{SZ02}), but also Pohozaev identity to handle the boundary term (see e.g. \cite{QX15,GW20,GL23}). Exploiting an
approximation argument (see e.g. \cite{CFR20,CM18}) and using these two integral
identities, we finally show that $u$ must be of the form $(\ref{e1.1})$ in Section 3.

Secondly, we consider the following equations
\begin{equation}\label{e02}
    \left\{
    \begin{array}{lr}
	\Delta_p u+d u^{p^*-1}=0 & \T{in} \quad \mathbb{R}^n_+\\
	|D u|^{p-2}\frac{\partial u}{\partial t}=\mp u^q &\T{on} \quad \p\mathbb{R}^n_+ \\
    u>0 & \T{in} \quad \mathbb{R}^n_+,
	\end{array}
 \right.     
\end{equation}
where $p\in(1,n),\,d>0$. We can see $(\ref{e02})$ are the Euler-Lagrange equations of the following  interpolation inequality between Sobolev inequality and Sobolev trace inequality (see e.g. \cite{BL85,MNT23})
\begin{equation}
\label{e03}\frac{\|\n\vp\|_{L^p(\mathbb{R}^n_+)}^p}{A}\mp\frac{\|\vp\|_{L^{p_*}(\p\mathbb{R}^n_+)}^p}{B}
\ge \|\vp\|_{L^{p^*}(\mathbb{R}^n_+)}^p
\end{equation}
for some appropriate $A>0,B>0$ and every $\vp\in C_c^\infty(\ov{\mathbb{R}^n_+})$.

Adapting the proof of theorem \ref{t1}, we get the following theorem in Section 4, which implies that every positive critical point of $(\ref{e03})$ must be of the form $(\ref{e04})$.
\begin{thm}\label{t02}
Let $n\ge 2,\,1<p<n,\,d>0$ and let $u\in D^{1,p}(\mathbb{R}^n_+)$ be a solution to equation $(\ref{e02})$. Then 
\begin{equation}
    \label{e04}u(x)=\bigg(\frac{n-p}{p}\bigg)^{\frac{n-p}{p}}\bigg(\frac{\frac{n(p-1)}{p}\l^{\frac{p}{p-1}}|x-x_0|^{\frac{p}{p-1}}+d\frac{n-p}{p}}{n\l}\bigg)^{-\frac{n-p}{p}}
\end{equation}
for some $x_0\in \rz^{n}_{\mp}$ and $\l=\mp\frac{1}{x_0^n}>0$.
\end{thm}

Under an additional hypothesis $u(x) = O(|x|^{2-n})$ for large $|x|$, the result for $p=2$ was 
established earlier by Escobar (\cite{E90}) along the lines of the 
proof of Obata. Later, using the method of moving spheres, 
a variant of the method of moving planes, \cite{LZ95} obtained classification results for equation $(\ref{e02})$ when $p=2$.

\noindent\textbf{Structure of the paper.} The frame of the paper can be explained as follows. First, in Section 2 we show that $u$ is bounded and satisfies some decay estimates at infinity. Hence, by elliptic regularity theory for p-Laplacian type equations \cite{L88,L91}, we can get an asymptotic upper bound on $|\n u|$ (in particular arguing as \cite{Veto16}, we can see
it behaves as the fundamental solution both from above and below). Also, using the method like \cite{CFR20}, we prove in Subsection 2.3 that $|\n u|^{p-2}\n u\in W_{loc}^{1,2}$ and we get an asymptotic integral estimate on second order derivatives.\\
Then, in Section 3 we consider the function $v=u^{-\frac{p}{n-p}}$ and use some integral identities to prove that $\n(|\n v|^{p-2}\n v)$ is a multiple of the identity matrix, from which the symmetry result follows. Finally, following the proof of Theorem \ref{t1}, we shall get Theorem \ref{t02} in Section 4.

{\bf Acknowledgment:}
I am grateful to Prof. Xi-Nan Ma for his advanced guidance. I would like to thank Dao-Wen Lin and Wang-Zhe Wu for some helpful discussions and useful suggestions. Also, I'd like to thank Prof. Jingbo Dou for his help in revising this paper. This work was supported by National Natural Science Foundation of China [grant number 12141105].

\section{PRELIMINARIES}

In the whole paper, we denote by $B_r(x)$ the usual Euclidean ball, by $B_r$ the ball $B_r(0)$ centered at the origin, and by $\Omega$ the upper half space $\mathbb{R}^n_+:=\mathbb{R}^{n-1}\times\{t>0\}$.\\
Also, we denote  $p^*=\frac{n p}{n-p}$, $p_*=\frac{(n-1) p}{n-p}$, $b:=\frac{n-1}{n-p}$, $a:=\frac{n}{n-p}$, then $q=p_{*}-1$, $p_*=p b$ and $p^*=p a$.\\
Besides, we always assume that $\eta\in C_c^\infty(\rz^n)$ is a cut-off function such that $0\le \eta\le 1$.

\begin{obs}\label{o1}
    If $u\in D^{1,p}(\O)$, then $u\in L^{p_*}(\p \O)$ and $(\ref{e2})$ also holds for $u$.
\end{obs}
\begin{proof}
    Let $\eta$ be a non-negative cut-off function such that $\eta=1$ in $B_R$, $\eta=0$ outside $B_{2R}$ and $|\n \eta|\le \frac{C_n}{R}$. Then applying Sobolev trace inequality with $u\eta$, we get
    \begin{align*}
        \|u\eta\|_{L^{p_*}(\p\O)}^p\le& C_{n,p}
   \bigg(\int_{\O} |\n u|^p\eta^p dx+ \int_{\O} u^p|\n\eta|^p dx\bigg)\\
    \le&C_{n,p}
   \bigg(\int_{\O} |\n u|^p\eta^p dx+C_{n,p}\|u\|_{L^{p^*}(\O\cap (B_{2 R}\bs B_R))}^p\bigg)
    \end{align*}
    Let $R\ra \infty$, we have $u\in L^{p_*}(\p \O)$ and Sobolev trace inequality $(\ref{e2})$ holds for any $u\in D^{1,p}(\O)$.
\end{proof}
\begin{rem}\label{r2.1}
    Like the proof in Observation \ref{o1}, we can see that $(\ref{er1})$ holds for any $\vp\in D^{1,p}(\O)$.
\end{rem}

\begin{obs}\label{o2}
    If $u\in W^{1,p}(\O\cap B_R(0))$ and $u|_{\p B_R(0)\cap\O}=0$, then \begin{equation}\label{e3}
       \|u\|_{L^{p^*}(\O\cap B_R(0))}\le C_{n,p}\|\n u\|_{L^p(\O\cap B_R(0))} .
    \end{equation}
\end{obs}
\begin{proof}
    Reflecting $u$, we get $$\ov u(y,t)=\begin{cases}
        u(y,t) &  , t\ge 0\\
        u(y,-t) &, t<0.
    \end{cases}$$
    Then applying Sobolev inequality with $\ov u\in W_0^{1,p}(B_R)$, we obtain Observation \ref{o2}.
\end{proof}
\begin{rem}\label{r2.1.5}
    Like the proof in Observation \ref{o1}, we can see that $(\ref{e3})$ also holds for any $u\in D^{1,p}(\O)$.
\end{rem}

\begin{obs}\label{o3}
    For $0<x<t<y,\,\e\in(0,1)$, we have $$\|u\|_{L^t}\le \e\|u\|_{L^x}+\e^{-\s} \|u\|_{L^y}$$
    where $\s=\frac{(y-t)x}{(t-x)y},\,u\in L^x\cap L^y$.
\end{obs}
\begin{proof}
    Write $t=\l x+(1-\l)y$, then $\l=\frac{y-t}{y-x}\in (0,1)$. By H\"{o}lder's inequality and Young's inequality, we get
    \begin{align*}   
    \int u^t=&\int u^{\l x+(1-\l)y}\le \bigg(\int u^x \bigg)^\l \bigg(\int u^y \bigg)^{1-\l}\\
    \RA \|u\|_{L^t}\le &\|u\|_{L^x}^{\frac{\l x}{t }}\|u\|_{L^y}^{\frac{(1-\l)y}{t}}\le \e\|u\|_{L^x}+\e^{-\s} \|u\|_{L^y}
    \end{align*}
    
\end{proof}

\subsection{Boundedness of solutions}
Using Moser iteration like \cite{Veto16,W10,MW19,CFR20}, we can prove that solutions to $(\ref{e1})$ are bounded. Indeed, the result holds for more
general Neumann problem.

\begin{lem}\label{l1}
    Suppose $u\in D^{1,p}(\O)$ is a solution to 
\begin{equation}\label{e4}
    \left\{
    \begin{array}{lr}
	\div (a(\n u))=0 & ,\T{in} \,\O\\
	a(\n u)\cdot \nu=\Phi(x,u) & ,x\in \p \O\\
    u>0 & ,\T{in} \,\O
	\end{array}
 \right.
\end{equation}
where $a:\rz^n\ra\rz^n$ is a continuous vector field such that the following holds: there exist $\a>0,\,\l\ge 0,\,\g \ge 0$ and $0\le s\le 1/2$ such that 
\begin{equation}\label{e5}
    \left\{
    \begin{array}{lr}
	|a(x)|\le \a (|x|^2+s^2)^{\frac{p-1}{2}}&,\,x\in \rz^n \\
	x\cdot a(x)\ge \frac{1}{\a}\int_0^1(t^2|x|^2+s^2)^{\frac{p-2}{2}}|x|^2 dt&,\,x\in \rz^n\\
    |\Phi(x,z)|\le \l |z|^q+\g&,\,x\in \rz^n,\,z\in \rz.
	\end{array}
 \right.
\end{equation}
Then there exist $\d>0$ with the following property: let $0<\rho\le c_n$, $x_0\in\O$ be such that 
\begin{align*}
    &\T Vol (B_{c_n})\le 1,\,\T Area (B_{c_n}\cap \p\O)\le 1,\\
    &\|u\|_{L^{p_*}(B_{\rho}(x_0)\cap \p\O)}\le \d,\,\|u\|_{L^{p^*}(B_{\rho}(x_0)\cap \O)}\le \d
\end{align*}
   then for any $t>0$
   $$\|u\|_{L^\infty(B_{\rho/4}(x_0)\cap \O)}\le C\big(\|u\|_{L^t(B_{\rho/2}(x_0)\cap \O)}+\|u\|_{L^t(B_{\rho/2}(x_0)\cap \p\O)}+1\big)$$
   where $C$ depend only on $n,p,\a,\rho,\l,t,\g$; and $\d$ depend only on $n,p,\a,\l$.
\end{lem}
\begin{proof}
    Following the Moser iteration argument in \cite[Thm4.1]{HL00} and \cite[lem2.1]{CFR20}, we give a sketch of the proof. Given $l>0,\,k>0$, we define
    $\ov u=\min\{u,l\}+k, u_k=u+k$, then
    $$\ov u,u_k\in W_{loc}^{1,p}(\ov\O),\,l+k\ge\ov u\ge k>0,\,\T{and}\,\,\n \ov u=\begin{cases}
        \n u&\T{in}\,\{u\le l\}\\
        0&\T{in}\,\{u>l\}
    \end{cases}$$
Let $\eta\in C_c^\infty(\rz^n),\,\b\ge  p_*-p+1$ and use $\vp=\eta^p(\ov u^{\b-1}u_k-k^\b)$ as a test function in $(\ref{e4})$, then we get
\begin{align}
    \int_\O a(\n u)\cdot\n(\eta^p(\ov u^{\b-1}u_k-k^\b)) dx=\int_{\p\O}& \Phi(x,u)\eta^p(\ov u^{\b-1}u_k-k^\b) d\s\non\\
    \RA (\b-1)\int_\O a(\n \ov u)\cdot\n \ov u \,\eta^p\ov u^{\b-1} d x+&\int_\O a(\n u)\cdot\n u \,\eta^p\ov u^{\b-1} d x\non\\
    =-p\int_\O a(\n u)\cdot\n\eta\,\eta^{p-1}(\ov u^{\b-1}u_k-k^\b) dx
    &+\int_{\p\O} \Phi(x,u)\eta^p(\ov u^{\b-1}u_k-k^\b) d\s\label{e6}
\end{align}
As in \cite[lem2.1]{CFR20}, $(\ref{e5})$ implies \begin{equation}\label{e7}
    x\cdot a(x)\ge C_\a (|x|^p-s^p)
\end{equation}
From $(\ref{e6}),\,(\ref{e7})$,  $(\ref{e5})$ and $s\in [0,1/2]$, we have
\begin{align}
    &(\b-1)\int_\O |\n \ov u |^p\,\eta^p\ov u^{\b-1} d x+\int_\O |\n u |^p\,\eta^p\ov u^{\b-1} d x\non\\
    &\le C_1\bigg(
    \int_\O |a(\n u)||\n\eta|\,\eta^{p-1}(\ov u^{\b-1}u_k-k^\b) dx+\b\int_{\O} s^p\eta^p\ov u^{\b-1} d x\non\\
    &+\int_{\p\O} (\l u^q+\g)\eta^p(\ov u^{\b-1}u_k-k^\b) d\s\bigg)\non\\
    &\le C_1\bigg(
    \int_\O |a(\n u)||\n\eta|\,\eta^{p-1}\ov u^{\b-1}u_k dx+\int_{\p\O} (\l u_k^q+\g)\eta^p\ov u^{\b-1}u_k d\s\non\\&+\b\int_{\O}\eta^p\ov u^{\b-1} d x\bigg)
    \label{e8}
\end{align}
where $C_1$ always depends only on $\a,p$.\\
   Note that $(a+b)^x\le \max\{1,2^{x-1}\}(a^x+b^x)$ for $a,b,x>0$. By Young's inequality and $(\ref{e5})$, for any $\e\in (0,1)$, we have
   \begin{align*}   
   |a(\n u)||\n\eta|\,\eta^{p-1}\ov u^{\b-1}u_k\le &\e^{\frac{p}{p-1}}|a(\n u)|^{\frac{p}{p-1}}\,\eta^p\ov u^{\b-1}+ \e^{-p}\ov u^{\b-1}u_k^p|\n \eta|^p\\
   \le& C_p\e^{\frac{p}{p-1}}(|\n u|^p+s^p)\,\eta^p\ov u^{\b-1}+ \e^{-p}\ov u^{\b-1}u_k^p|\n \eta|^p
   \end{align*}
Choosing $\e>0$ small enough which depends only on $\a,p$, and using the fact that $\b\ge p_*-p+1$, we deduce that
\begin{align}   
&(\b-1)\int_\O |\n \ov u |^p\,\eta^p\ov u^{\b-1} d x+\frac{1}{2}\int_\O |\n u |^p\,\eta^p\ov u^{\b-1} d x\non\\
&\le
C_1\bigg(
    \int_\O \ov u^{\b-1}u_k^p|\n \eta|^p d x+\int_{\p\O} (\l u_k^q+\g)\eta^p\ov u^{\b-1}u_k d\s\non\\&+\b\int_{\O}\eta^p\ov u^{\b-1} d x\bigg)\non\\
    \RA& \int_\O |\n(\eta \ov u^{\frac{\b-1}{p}}u_k)|^p d x\le
    C_1\bigg(
    \b^{p-1}\int_\O \ov u^{\b-1}u_k^p|\n \eta|^p d x\non\\&+\b^{p-1}\int_{\p\O} (\l u_k^q+\g)\eta^p\ov u^{\b-1}u_k d\s+\b^p\int_{\O}\eta^p\ov u^{\b-1} d x\bigg)
\non
\end{align}
Since $\b\ge p_*-p+1,\, u_k\ge k$, we get
\begin{align}
    \int_\O |\n(\eta \ov u^{\frac{\b-1}{p}}u_k)|^p d x\le&
    C_1\bigg(\b^{p-1}
    \int_\O \ov u^{\b-1}u_k^p|\n \eta|^p d x+(\frac{\b}{k})^p\int_{\O}\eta^p\ov u^{\b-1}u_k^p d x\non\\+&\g(\frac{\b}{k})^{p-1}\int_{\p\O}\eta^p\ov u^{\b-1}u_k^p d\s+\l\b^{p}\int_{\p\O}\eta^p\ov u^{\b-1}u_k^{q+1} d\s\bigg)\label{e9}
\end{align}
where $C_1$ always depends only on $\a,p$.\\
Note $\eta^p\ov u^{\b-1}u_k^{q+1}= \eta^p\ov u^{\b-1}u_k^p\cdot u_k^{p_*-p}$. If we choose $\eta$ such that $\operatorname{supp}(\eta)\subset B_\rho(x_0)$, it follows from H\"{o}lder's inequality and Sobolev trace inequality that
\begin{align}  
\int_{\p\O}\eta^p\ov u^{\b+q-1}u_k d\s\le
&\| u_k\|_{L^{p_*}(B_{\rho}(x_0)\cap \p\O)}^{\frac{p(p-1)}{n-p}}\,\|\eta \ov u^{\frac{\b-1}{p}}u_k\|_{L^{p_*}(\p\O)}^p\non\\
\le
&(\|u\|_{L^{p_*}(B_{\rho}(x_0)\cap \p\O)}+k)^{\frac{p(p-1)}{n-p}}\|\eta \ov u^{\frac{\b-1}{p}}u_k\|_{L^{p_*}(\p\O)}^p\non\\
\le& C_{n,p} \,(\d+k)^{\frac{p(p-1)}{n-p}}\int_\O |\n(\eta \ov u^{\frac{\b-1}{p}}u_k)|^p d x\label{e10}
\end{align}
First, we fix $\b=p_*-p+1,\,k=\d$ and choose $\d>0$ small enough which depends only on $\l,n,p,\a$. 
From $(\ref{e9}),(\ref{e10})$, we get
\begin{align}
    \int_\O |\n(\eta \ov u^{\frac{\b-1}{p}}u_k)|^p d x\le&
    C_2\bigg(\b^{p-1}
    \int_\O \ov u^{\b-1}u_k^p|\n \eta|^p d x+\b^p\int_{\O}\eta^p\ov u^{\b-1}u_k^p d x\non\\&+\g\b^{p}\int_{\p\O}\eta^p\ov u^{\b-1}u_k^p d\s\bigg)\non
\end{align}
where $C_2$ always depends only on $\a,n,p,\l$.\\
Hence, thanks to the Sobolev trace inequality and Observation \ref{o2}, we obtain 
\begin{align}
    \|\eta^p\ov u^{\b-1}u_k^p\|_{L^a(\O)}+&\|\eta^p\ov u^{\b-1}u_k^p\|_{L^b(\p\O)}\le
    C_2\bigg(\b^{p-1}
    \int_\O \ov u^{\b-1}u_k^p|\n \eta|^p d x\non\\&+\b^p\int_{\O}\eta^p\ov u^{\b-1}u_k^p d x+\g\b^{p}\int_{\p\O}\eta^p\ov u^{\b-1}u_k^p d\s\bigg)\non
\end{align}
Since $\rho\le c_n,\,\T Vol (B_{c_n})\le 1,\,\T Area (B_{c_n}\cap \p\O)\le 1$, here we choose $\eta$ such that $\eta=1$ in $B_{\rho/2}(x_0)$ and $|\n \eta|\le C_n/\rho$, then by H\"{o}lder's inequality and $u_k\ge\ov u$, we have
\begin{align}
    &\|\ov u^{p_*-p}u_k^p\|_{L^a(\O\cap B_{\rho/2}(x_0))}+\|\ov u^{p_*-p}u_k^p\|_{L^b(\p\O\cap B_{\rho/2}(x_0))}\non\\
    &\le C_3\bigg(\| u_k^{p_*}\|_{L^{\frac{a}{b}}(\O\cap B_{\rho}(x_0))}
+\g\|u_k^{p_*}\|_{L^{1}(\p\O\cap B_{\rho}(x_0))}\bigg)\non\\
&\le C_3\bigg((\|u\|_{L^{p^*}(\O\cap B_{\rho}(x_0))}+k)^{p_*}
+\g(\| u\|_{L^{p_*}(\p\O\cap B_{\rho}(x_0))}+k)^{p_*}\bigg)\non
\end{align}
So, letting $l\ra \infty$, we get
\begin{align}
\|u_k\|_{L^{p_* b}(\p\O\cap B_{\rho/2}(x_0))}\le &C_3\bigg(\|u\|_{L^{p^*}(\O\cap B_{\rho}(x_0))}+\g^{\frac{1}{p_*}}\| u\|_{L^{p_*}(\p\O\cap B_{\rho}(x_0))}+\d\bigg)\label{e11}\\
\le &C_3(\d+\d\g^{\frac{1}{p_*}} ):=C_4\non
\end{align}
where $C_3$ always depends only on $\a,n,p,\l,\rho$.\\
Second, let $\b$ vary, choose $\eta$ such that $\operatorname{supp}(\eta)\subset B_{\rho/2}(x_0)$ and fix $k=\d$, we come back to $(\ref{e9})$. Like $(\ref{e10})$, we have
\begin{align}  
\int_{\p\O}\eta^p\ov u^{\b-1}u_k^{q+1} d\s\le
&\| u_k\|_{L^{p_* b}(B_{\rho/2}(x_0)\cap \p\O)}^{\frac{p(p-1)}{n-p}}\,\|\eta^p \ov u^{\b-1}u_k^p\|_{L^{t_0}(\p\O)}\non\\
\le
&C_4^{\frac{p(p-1)}{n-p}}\|\eta^p \ov u^{\b-1}u_k^p\|_{L^{t_0}(\p\O)}\label{e12}
\end{align}
where $\frac{1}{t_0}+\frac{b-1}{b^2}=1\RA 1<t_0<b=\frac{n-1}{n-p}$.\\
Then from Observation \ref{o3}, Sobolev trace inequality and $(\ref{e12})$, we obtain
\begin{align}  
\int_{\p\O}\eta^p\ov u^{\b-1}u_k^{q+1} d\s
\le
&C_4^{\frac{p(p-1)}{n-p}}\|\eta^p \ov u^{\b-1}u_k^p\|_{L^{t_0}(\p\O)}\non\\
\le& \e \b^{-p} \|\eta^p \ov u^{\b-1}u_k^p\|_{L^{b}(\p\O)}+ C_5\b^{\s}\|\eta^p\ov u^{\b-1}u_k^p\|_{L^{1}(\p\O)}\non\\
\le& C_{n,p}\e \b^{-p} \int_\O |\n(\eta \ov u^{\frac{\b-1}{p}}u_k)|^p d x+ C_5\b^{\s}\|\eta^p\ov u^{\b-1}u_k^p\|_{L^{1}(\p\O)}
\label{e13}
\end{align}
where $\s=\frac{p(b-t_0)}{b(t_0-1)}=p(b-1)$ and $C_5\sim C_4,\e$.\\
Then we choose $\e>0$ small enough depending only on $n,p,\a,\l$. From $(\ref{e9}),\,(\ref{e13})$, we get
\begin{align}
    \int_\O |\n(\eta \ov u^{\frac{\b-1}{p}}u_k)|^p d x\le&
    C_5\bigg(\b^{p-1}
    \int_\O \ov u^{\b-1}u_k^p|\n \eta|^p d x+\b^p\int_{\O}\eta^p\ov u^{\b-1}u_k^p d x\non\\&+(\g\b^{p}+\b^{p(b-1)})\int_{\p\O}\eta^p\ov u^{\b-1}u_k^p d\s\bigg)\label{e14}
\end{align}
Also, thanks to the Sobolev trace inequality and Observation \ref{o2}, we get
\begin{align}
    \|\eta^p\ov u^{\b-1}u_k^p&\|_{L^a(\O)}+\|\eta^p\ov u^{\b-1}u_k^p\|_{L^b(\p\O)}\le
    C_5\bigg(\b^{p-1}
    \int_\O \ov u^{\b-1}u_k^p|\n \eta|^p d x+\non\\&\b^p\int_{\O}\eta^p\ov u^{\b-1}u_k^p d x+(\g\b^{p}+\b^{p(b-1)})\int_{\p\O}\eta^p\ov u^{\b-1}u_k^p d\s\bigg)\non
\end{align}
Since $\rho\le c_n,\,\T Vol (B_{c_n})\le 1,\,\T Area (B_{c_n}\cap \p\O)\le 1$, here we choose $\eta$ such that $\eta=1$ in $B_{r}(x_0)$, $\eta=0$ outside $B_{R}(x_0)$ and $|\n \eta|\le C_n/(R-r)$, where $0<r<R<\rho/2$. Then by H\"{o}lder's inequality we have
\begin{align}
    \|\ov u^{\b-1}u_k^p\|_{L^b(\O\cap B_r(x_0))}+\|\ov u^{\b-1}u_k^p&\|_{L^b(\p\O\cap B_r(x_0))}\le
    \non\\C_5\bigg(\frac{1}{(R-r)^p}\b^p\|\ov u^{\b-1}u_k^p\|_{L^1(\O\cap B_R(x_0))}
+&(\g+1)\b^\tau\|\ov u^{\b-1}u_k^p\|_{L^1(\p\O\cap B_R(x_0))}\bigg)\label{e15}
\end{align}
where $\tau=\max\{p,p(b-1)\}$, $C_5$ depend only on $\a,n,p,\l,\rho,\g$.\\
Hence, as in the proof of \cite[Thm4.1]{HL00}, a classical Moser iteration argument yields the following result
\begin{equation}   \| u\|_{L^\infty(B_{\rho/4}(x_0)\cap \O)}\le C_5\big(\| u\|_{L^{p_*}(B_{\rho/2}(x_0)\cap \O)}+\| u\|_{L^{p_*}(B_{\rho/2}(x_0)\cap \p\O)}+1\big).\label{e16}
\end{equation}
    Then following the argument in \cite[Remark 4.2]{HL00}, we get the result.
\end{proof}
\begin{rem}\label{r2.2}
    Note if $s=\g=0$, we can choose $k=0$, and then $(\ref{e9})$ becomes
    \begin{align}
    \int_\O |\n(\eta \ov u^{\frac{\b-1}{p}}u)|^p d x\le
    C_1\bigg(
    \int_\O \ov u^{\b-1}u^p|\n \eta|^p d x+\l\b^{p}\int_{\p\O}\eta^p \ov u^{\b-1}u^{q+1} d\s\bigg)\non
\end{align}
and 
$(\ref{e11})$ becomes
\begin{align}
   \|\ov u\|_{L^{p_* b}(\O\cap B_{\rho/2}(x_0))}+\|\ov u\|_{L^{p_* b}(\p\O\cap B_{\rho/2}(x_0))}&\le C_3\|u\|_{L^{p^*}(\O\cap B_{\rho}(x_0))}.\non
\end{align}
So $(\ref{e16})$ becomes   $$\|u\|_{L^\infty(B_{\rho/4}(x_0)\cap \O)}\le C\|u\|_{L^{p_*}(B_{\rho/2}(x_0)\cap \O)}$$
   and then Lemma \ref{l1} becomes: for any $t>0$
   $$\|u\|_{L^\infty(B_{\rho/4}(x_0)\cap \O)}\le C\|u\|_{L^t(B_{\rho/2}(x_0)\cap \O)}.$$
\end{rem}

\subsection{Asymptotic bounds on $u$ and $\n u$}
Imitating the proof in \cite[Thm1.1]{Veto16} and \cite{L19}, we can prove the following proposition.

\begin{prop}\label{p1}
   Let $n\ge 2,\,1<p<n,$ and let $u$ be a solution to $(\ref{e1})$. Then there exist  $C_0,\,C_1,\,\a>0$ such that 
   \begin{align}
       u(x)\le &\frac{C_1}{1+|x|^\frac{n-p}{p-1}}\,,\,\,|\n u(x)|\le \frac{C_1}{1+|x|^\frac{n-1}{p-1}}\,\,\,\T{in} \,\,\,\O\label{e17.1}\\
       &u(x)\ge\frac{C_0}{1+|x|^\frac{n-p}{p-1}}\quad\T{in} \quad\O\label{e17.2}\\
       \big||x|^\frac{n-p}{p-1}u(x)&- \a\big|+\big||x|^\frac{n-1}{p-1}\n u(x)+\a\frac{n-p}{p-1}x\big|\ra 0\,\,\T{as}\,\,|x|\ra\infty\label{e17.3}
   \end{align}
\end{prop}

From Observation \ref{o1} and Lemma \ref{l1}, we get $\|u\|_{L^\infty(\O)}\le C$. Hence, by elliptic regularity theory for p-Laplacian type equations \cite{L88,L91}, we obtain
$u\in C^{1,\a}_{loc}(\ov \O)$. Therefore, to prove $(\ref{e17.1})$ and $(\ref{e17.2})$, it suffices to show there exists $R_0>1$ such that 
\begin{equation}
    C_0\le |x|^\frac{n-p}{p-1}u(x)\le C_1,\,\,|x|^\frac{n-1}{p-1}|\n u(x)|\le C_1\,\,\T{in}\,\,\O\cap \{|x|>R_0\}\label{e18}
\end{equation}
For $R>1,\,x\in \O$, we define 
\begin{equation}
    u_R(x):=R^\frac{n-p}{p-1}u(R x)\label{e19}
\end{equation}
From $(\ref{e1})$, we have
\begin{equation}\label{e20}
    \left\{
    \begin{array}{lr}
	\Delta_p u_R=0 & , \T{in} \,\,\O\\
	|D u_R|^{p-2}\frac{\partial u_R}{\partial t}=-R^{-1}u_R^q & , \T{on} \,\,\p\O
	\end{array}
 \right.     
\end{equation}
In order to prove $(\ref{e18})$, it suffices to show there exists $R_1>1$ such that if $R>R_1$, then
\begin{equation}
    C_0\le u_R(x)\le C_1,\,\,|\n u_R(x)|\le C_1\,\,\T{in}\,\,\O\cap \{4<|x|<5\}\label{e21}
\end{equation}
Denote by $V(x):=R^{-1}u_R^{p_*-p}(x)=R^{p-1}u^{p_*-p}(R x),\,U:=\{1<|x|<8\}$, then 
$$\|V\|_{L^\frac{n-1}{p-1}(U\cap\p\O)}=\|u\|_{L^{p_*}(\p\O\cap\{R<|x|<8 R\})}^{\frac{p(p-1)}{n-p}}\ra 0\,\,\T{as}\,\,R\ra \infty.$$
Denote by $W_1:=\{2<|x|<7\}$, $W:=\{3<|x|<6\}$.

\subsubsection{Asymptotic upper bounds on $u$ and $\n u$}
Using the method in \cite{L19}, we shall get asymptotic upper bounds on $u$.
\begin{lem}\label{l2}
Suppose $u\in W^{1,p}\cap L^\infty(U\cap \O )$ is a positive solution to
    \begin{equation}\label{e22}
    \left\{
    \begin{array}{lr}
	\Delta_p u=0 & ,\T{in} \,\,\O\cap U\\
	|D u|^{p-2}\frac{\partial u}{\partial t}=-f(x) u^{p-1} & ,\T{on} \,\,\p\O\cap U
	\end{array}
 \right.     
\end{equation}
where $f\in L^{\frac{n-1}{p-1}}(U\cap\p\O)$, $|f(x)|\le \l|u(x)|^{p_*-p}$ in $U\cap\p\O$ and $\l\ge 0$.\\
Then there exist $\d>0$ with the following property: let $c_n\in(0,1),\,\l$ be such that 
\begin{align*}
    \T Vol (B_{c_n})\le 1,\,\,\T Area (B_{c_n}\cap \p\O)\le 1,
    \,\,\l \|u\|_{L^{p_*}(U\cap\p\O)}^\frac{p(p-1)}{n-p}\le \d,\,\,\l \|u\|_{L^{p^*}(U\cap\O)}^\frac{p(p-1)}{n-p}\le \d,
\end{align*}
   then for any $t>0$, $0<\rho\le c_n$, $x_0\in\O\cap W$
   $$\|u\|_{L^\infty(B_{\rho/2}(x_0)\cap \O)}\le C\|u\|_{L^t(B_{\rho}(x_0)\cap \O)}$$
   where $C$ depend only on $n,p,\rho,t$; and $\d$ depend only on $n,p$.
\end{lem}
\begin{rem}\label{r2.3}
    The proof of Lemma \ref{l2} is similar to the one of Lemma \ref{l1}.
\end{rem}
\begin{proof}
    As in the proof in lemma \ref{l1}, the proof is based on a Moser iteration argument. Let $\eta\in C_c^\infty(U),\,\b\ge 1$ and use $\vp=\eta^p u^\b$ as a test function in $(\ref{e22})$, then we get
\begin{align}
     \b\int_\O |\n u|^p\,\eta^p u^{\b-1} d x=&-p\int_\O |\n u|^{p-2}\n u \cdot\n\eta\,\eta^{p-1} u^\b dx\non\\&+\int_{\p\O} f(x)\eta^p  u^{\b+p-1} d\s\non
     \end{align}
     By H\"{o}lder's inequality, we get
     \begin{align}
     \b\int_\O |\n u|^p\,\eta^p u^{\b-1} d x&\le C_1\bigg(\b^{-(p-1)}\int_\O |\n \eta|^p u^{\b+p-1} dx\non\\&+\int_{\p\O} |f(x)|\eta^p  u^{\b+p-1} d\s\bigg)\non\\
     \RA\int_\O |\n(\eta  u^{\frac{\b+p-1}{p}})|^p d x&\le C_1\bigg(\int_\O |\n \eta|^p u^{\b+p-1} dx\non\\&+\b^p\int_{\p\O} |f(x)|\eta^p  u^{\b+p-1} d\s\bigg)
     \label{e23}
\end{align}
where $C_1$ always depend only on $p$ and $(\ref{e23})$ is similar to $(\ref{e9})$.\\
Note that 
$$
\|f\|_{L^{\frac{n-1}{p-1}}(U\cap \p\O)}\le \l \|u\|_{L^{p_*}(U\cap\p\O)}^\frac{p(p-1)}{n-p}\le \d
$$
So, like $(\ref{e10})$ it follows from H\"{o}lder's inequality and Sobolev trace inequality that
\begin{align}  
\int_{\p\O}|f(x)|\eta^p  u^{\b+p-1} d\s\le
&\|f\|_{L^{\frac{n-1}{p-1}}(U\cap \p\O)}\,\|\eta u^{\frac{\b+p-1}{p}}\|_{L^{p_*}(\p\O)}^p\non\\
\le& C_{n,p} \,\d\int_\O |\n(\eta u^{\frac{\b+p-1}{p}})|^p d x\label{e24}
\end{align}
First, we fix $\b=p_*-p+1$ and choose $\d>0$ small enough which depends only on $n,p$. 
From $(\ref{e23}),(\ref{e24})$, we get
\begin{align}
    \int_\O |\n(\eta \ov u^{\frac{\b+p-1}{p}})|^p d x\le&
    C_2\int_\O |\n \eta|^p u^{\b+p-1} dx\non
\end{align}
where $C_2$ always depends only on $n,p$.\\
Hence, thanks to the Sobolev trace inequality and Observation \ref{o2}, we obtain 
\begin{align}
    \|\eta^p u^{p_*}\|_{L^a(\O)}+\|\eta^p u^{p_*}\|_{L^b(\p\O)}\le
    C_2\int_\O |\n \eta|^p u^{p_*} dx\non
\end{align}
    Here we choose $\eta$ such that $\eta=1$ in $W_1$, $\eta=0$ outside $U$ and $|\n \eta|\le C_n$, then by H\"{o}lder's inequality we have
    \begin{align}
        &\| u\|_{L^{p_* b}(\p\O\cap W_1)}\le C_2\| u\|_{L^{p_*}(\O\cap U)}\le C_2\| u\|_{L^{p^*}(\O\cap U)}\non\\
        \RA \|f&\|_{L^{\frac{n-1}{p-1}b}(W_1\cap \p\O)}\le
        \l \|u\|_{L^{p_* b}(W_1\cap\p\O)}^\frac{p(p-1)}{n-p}\le C_2 \l \|u\|_{L^{p^*}(U\cap\O)}^\frac{p(p-1)}{n-p}\le C_2\label{e25}
    \end{align}
Second, let $\b$ vary and $\operatorname{supp}\eta\subset W_1$, we come back to $(\ref{e23})$. Like $(\ref{e24})$, we have
\begin{align}  
\int_{\p\O}|f(x)|\eta^p  u^{\b+p-1} d\s\le
&\|f\|_{L^{\frac{n-1}{p-1}b}(W_1\cap \p\O)}\,\|\eta^p u^{\b+p-1}\|_{L^{t_0}(\p\O)}\non\\
\le& C_2\|\eta^p u^{\b+p-1}\|_{L^{t_0}(\p\O)}\label{e26}
\end{align}
where $\frac{1}{t_0}+\frac{b-1}{b^2}=1\RA 1<t_0<b=\frac{n-1}{n-p}$.\\
Then from Observation \ref{o3}, Sobolev trace inequality and $(\ref{e26})$, we obtain
\begin{align}  
\int_{\p\O}|f(x)|\eta^p  u^{\b+p-1} d\s
\le& \e \b^{-p} \|\eta^p u^{\b+p-1}\|_{L^{b}(\p\O)}+ C_3\b^{\s}\|\eta^p u^{\b+p-1}\|_{L^{1}(\p\O)}\non\\
\le C_{n,p}\e &\b^{-p} \int_\O |\n(\eta u^{\frac{\b+p-1}{p}})|^p d x+ C_3\b^{\s}\|\eta^p u^{\b+p-1}\|_{L^{1}(\p\O)}
\label{e27}
\end{align}
where $\s=\frac{p(b-t_0)}{b(t_0-1)}=p(b-1)$ and $C_3\sim n,p,\e$.\\
    Then we choose $\e>0$ small enough depending only on $n,p$. From $(\ref{e23}),\,(\ref{e27})$, we get
\begin{align}
\int_\O |\n(\eta  u^{\frac{\b+p-1}{p}})|^p d x&\le C_2\bigg(\int_\O |\n \eta|^p u^{\b+p-1} d x\non\\&+\b^\s\int_{\p\O} \eta^p  u^{\b+p-1} d\s\bigg)
     \label{e28}
\end{align}
Also, thanks to the Sobolev trace inequality and Observation \ref{o2}, we get
   \begin{align}
    \|\eta^p u^{\b+p-1}\|_{L^a(\O)}+\|\eta^p u^{\b+p-1}\|_{L^b(\p\O)}&\le C_2\bigg(\int_\O |\n \eta|^p u^{\b+p-1} d x\non\\&+\b^\s\int_{\p\O} \eta^p  u^{\b+p-1} d\s\bigg)
     \label{e29}
    \end{align}
Since $\rho\le c_n,\,\T Vol (B_{c_n})\le 1,\,\T Area (B_{c_n}\cap \p\O)\le 1$, here we choose $\eta$ such that $\eta=1$ in $B_{r}(x_0)$, $\eta=0$ outside $B_{R}(x_0)$ and $|\n \eta|\le C_n/(R-r)$, where $0<r<R<\rho$. Then by H\"{o}lder's inequality we have
\begin{align}
    \|u^{\b+p-1}\|_{L^b(\O\cap B_r(x_0))}+\|u^{\b+p-1}&\|_{L^b(\p\O\cap B_r(x_0))}\le
    \non\\C_2\bigg(\frac{1}{(R-r)^p}\| u^{\b+p-1}\|_{L^1(\O\cap B_R(x_0))}
+&\b^\s\| u^{\b+p-1}\|_{L^1(\p\O\cap B_R(x_0))}\bigg)\label{e30}
\end{align}
where $\s=p(b-1)$.\\
Hence, as in the proof of \cite[Theorem 4.1,Remark 4.2]{HL00}, a classical Moser iteration argument yields the result.
   
\end{proof}

\begin{rem}\label{r2.4}
    In our setting, $\l=R^{-1}$, $f(x)=V(x)$. Note that 
   \begin{align*} 
    R^{-\frac{n-p}{p(p-1)}}\|u_R\|_{L^{p_*}(U\cap\p\O)}=\|u\|_{L^{p_*}(\p\O\cap\{R<|x|<8 R\})}\ra 0\,\,\T{as}\,\,R\ra \infty,
   \end{align*}
   and
   \begin{align*} 
    R^{-\frac{n-p}{p(p-1)}}\|u_R\|_{L^{p^*}(U\cap\O)}=\|u\|_{L^{p^*}(\O\cap\{R<|x|<8 R\})}\ra 0\,\,\T{as}\,\,R\ra \infty.
   \end{align*}
    Then, by Lemma \ref{l2} we get there exists $R_2>1$ such that
    $$
    \|u_R\|_{L^\infty(W\cap \O)}\le C\|u_R\|_{L^{p^*}(W_1\cap \O)}\le C R^{\frac{n-p}{p(p-1)}}\,\,
    ,\T{for}\,\,R\ge R_2
    $$
Hence, by elliptic regularity theory for p-Laplacian type equations \cite{L88,L91}, for any $G\Subset W$, we obtain $\|u_R\|_{C^{1,\a}(G\cap \O)}\le C R^{\frac{n-p}{p(p-1)}}$. Therefore, we have
$$
|x|^\frac{n-p}{p}u(x)+|x|^\frac{n}{p}|\n u(x)|\le C\,\,\T{in}\,\,\O\cap \{|x|>R_2\}
$$

\end{rem}

It's known that for a n-dimensional set $\O\subset\rz^n$, $t>0$ and $v\in L^t(\O)$, we have 
\begin{equation}\label{tool1}
    \int_{\O} |v|^t d x=t\int_0^\infty s^{t-1}\H^{n}(\O\cap\{|v|>s\}) d s
\end{equation}
Recall that, for a n-dimensional set $\O\subset\rz^n$ and $r>0$, one defines the space $L^{r,\infty}(\O)$ as the set of all measurable functions $v:\O\ra\rz$ such that
$$
\|v\|_{L^{r,\infty}(\O)}:=\operatorname*{sup}_{h>0}\,\{h \,\H^n(\{|v|>h\})^\frac{1}{r}\}<\infty
$$
where $\H^k$ is k-dimensional hausdorff measure. 

\begin{lem}\label{l3}
    Every solution of $(\ref{e1})$ belongs to $L^{(p-1)a,\infty}(\O)$ and $L^{(p-1)b,\infty}(\p\O)$.
\end{lem}

\begin{proof}
    Adapting the proof in \cite[lem2.2]{Veto16}, we obtain this lemma.\\    
\textbf{Step 1:}\textit{ Let $u$ be a solution of $(\ref{e1})$, then $u\in L^{(p-1)b,\infty}(\p\O)$.}\\
For every $h>0$, we define $u_h=\min\{u,h\}$, $W_h=\{u>h\}$, then 
$$u_h\in D^{1,p}(\O)\,\,\,\,\T{and}\,\,\,\,\n u_h=\begin{cases}
    \n u & \T{in} \,\,W_h^c\\
    0 & \T{in} \,\,W_h
\end{cases}$$
By Remark \ref{r2.1}, we use $\vp=u_h$ as a test function in $(\ref{e1})$, then we get
\begin{align}
    \int_{\O} |\n u_h|^p d x=&\int_{\p\O} u^q u_h d\s\non\\
    =&\int_{\p\O\cap W_h^c} u^{p_*} d\s+h\, \int_{\p\O\cap W_h} u^q d\s\label{e33}
\end{align}
Thanks to the Observation \ref{o1} and Observation \ref{o2}, we have 
\begin{align}
\bigg(\int_{\O}u_h^{p^*} d x\bigg)^\frac{1}{a}+&\bigg(\int_{\p\O}u_h^{p_*} d \s\bigg)^\frac{1}{b}\le C_1 \int_{\O} |\n u_h|^p d x\non\\
&\le C_1\bigg(\int_{\p\O\cap W_h^c} u^{p_*} d\s+h\, \int_{\p\O\cap W_h} u^q d\s\bigg)\label{e33.5}
\end{align}
where $C_1$ always depend only on $n,p$.\\
On the one hand, straightforward computations give
\begin{align}\label{e34}
    \int_{\p\O\cap W_h^c} u^{p_*} d\s=\int_{\p\O}u_h^{p_*} d \s-h^{p_*}\,\H^{n-1}(W_h\cap\p\O)
\end{align}
On the other hand, using $(\ref{tool1})$ we get 
\begin{align}
    \int_{\p\O\cap W_h} u^q d\s=&q\int_0^\infty s^{q-1}\H^{n-1}(\{u>h\}\cap\{u>s\}\cap\p\O)\, d s\non\\
    =&h^q \H^{n-1}(W_h\cap\p\O)+q\int_h^\infty s^{q-1}\H^{n-1}(\{u>s\}\cap\p\O)\, d s\label{e35}
\end{align}
It follows from $(\ref{e33.5})$, $(\ref{e34}),\,(\ref{e35})$ and $q=p_*-1$ that
\begin{align}
    \bigg(\int_{\p\O}u_h^{p_*} d \s\bigg)^\frac{1}{b}\le C_1\bigg(
    \int_{\p\O}u_h^{p_*} d \s+q h\int_h^\infty s^{q-1}\H^{n-1}(\{u>s\}\cap\p\O)\, d s
    \bigg)\label{e36}
\end{align}
Since $|u_h|\le |u|,\,u\in L^{p_*}(\p\O),\,u_h\ra 0$ as $h\ra 0$, by dominated convergence theorem, we get
\begin{equation}\label{e37}
\int_{\p\O}u_h^{p_*} d \s\ra0\,\,\,\T{as}\,\,\,h\ra 0
\end{equation}
It follows from $(\ref{e34}),\,(\ref{e36})$ and $(\ref{e37})$ that for small $h>0$, we have 
\begin{align}
    h^{p_*}\,\H^{n-1}&(W_h\cap\p\O)\le \int_{\p\O}u_h^{p_*} d \s\le C_1\bigg(
    h\int_h^\infty s^{q-1}\H^{n-1}(\{u>s\}\cap\p\O)\, d s
    \bigg)^b\non\\
    \RA &h^{(p-1)b}\,\H^{n-1}(W_h\cap\p\O)\le C_1\bigg(
    \int_h^\infty s^{q-1}\H^{n-1}(\{u>s\}\cap\p\O)\, d s
    \bigg)^b\label{e38}
\end{align}
We now define 
$$
G(h):=\int_h^\infty s^{q-1}\H^{n-1}(\{u>s\}\cap\p\O)\, d s
$$
Then we have $G'(h)=-h^{q-1}\,\H^{n-1}(W_h\cap\p\O)\le 0$ when $h>0$, which implies $G(0):=\D\operatorname*{lim}_{h\ra 0^+}G(h)$ exists (but may be infinity). In the following, we will show that $G(0)<\infty$. \\
By using $(\ref{e38})$, we obtain
\begin{align}
    (G(h)^{1-b})'=&(b-1)G(h)^{-b} h^{q-1}\,\H^{n-1}(W_h\cap\p\O)\non\\
    \le& C_1 h^{b-2}\,\quad\quad\T{ for small}\,\,\,h>0.  \label{e39}
\end{align}
By integrating $(\ref{e39})$, we obtain
\begin{align}
    G(h)^{1-b}-G(0)^{1-b}
    \le C_1 h^{b-1}\,\quad\quad\T{ for small}\,\,\,h>0.  \label{e40}
\end{align}
Note that \begin{equation*}
    h\,G(h)=\int_0^\infty h s^{q-1} \H^{n-1}(\{u>s\}\cap\p\O)\,\chi_{\{s>h\}}\, d s.
\end{equation*}
Since $|h\,s^{q-1}\H^{n-1}(\{u>s\}\cap\p\O)\,\chi_{\{s>h\}}|\le s^{q}\,\H^{n-1}(\{u>s\}\cap\p\O)\in L^1(\rz_+)$ and $|h\,s^{q-1}\H^{n-1}(\{u>s\}\cap\p\O)\,\chi_{\{s>h\}}|\ra 0$ as $h\ra 0$, by dominated convergence theorem we get
\begin{equation}\label{e41}
h\,G(h)\ra0\,\,\,\T{as}\,\,\,h\ra 0.
\end{equation}
It follows from $(\ref{e40})$ and $(\ref{e41})$ that $G(0)<\infty$. By using $(\ref{e38})$ and $G$ is non-increasing, for small $h>0$ we get
\begin{align}
h^{(p-1)b}\,\H^{n-1}(W_h\cap\p\O)\le C_1 G(0)^b<\infty.\non
\end{align}
    As for $h>1$, it follows from $(\ref{e34})$ that
\begin{align}
&h^{(p-1)b}\,\H^{n-1}(W_h\cap\p\O)\le h^{p_*}\,\H^{n-1}(W_h\cap\p\O)\non\\
&\le \int_{\p\O}u_h^{p_*} d \s
\le \int_{\p\O}u^{p_*} d \s
<\infty.\non
\end{align}
In conclusion, we obtain $u\in L^{(p-1)b,\infty}(\p\O)$.\\
\textbf{Step 2:}\textit{ Let $u$ be a solution of $(\ref{e1})$, 
then $u\in L^{(p-1)a,\infty}(\O)$.}\\

Come back to (\ref{e33.5}), we have
\begin{align}
    \bigg(\int_{\O}u_h^{p^*} d x\bigg)^\frac{1}{a}\le C_1\bigg(
    \int_{\p\O}u_h^{p_*} d \s+q h\, G(h)\bigg)\non
\end{align}
Note that from (\ref{e38}) and step 1, for $h>0$ small we get
$$\int_{\p\O}u_h^{p_*} d \s\le C_1 (h G(h))^b \le C h^b \quad\T{and}\quad h G(h)\le C h. $$
Therefore, for $h>0$ small we have 
\begin{align}
   h^{p^*}\,\H^{n}(W_h\cap\O)\le\int_{\O}u_h^{p^*} d x\le C_1 h^a \label{e50}
\end{align}
It follows from $(\ref{e50})$ that
\begin{align}
    h^{(p-1)a}\,\H^{n}(W_h\cap\O)\le C_2 \quad\,\T{ for small}\,\,\,h>0.\label{e50.1}
\end{align}
 Similar to $(\ref{e34})$, for each $h>0$ we have
    \begin{align}
    h^{p^*}\,\H^{n}(W_h\cap\O)\le\int_{\O}u_h^{p^*} d x<\infty,\non
\end{align}
which implies that for $h>1$, we have
\begin{align}
    h^{(p-1)a}\,\H^{n}(W_h\cap\O)\le h^{p^*}\,\H^{n}(W_h\cap\O)<\infty\label{e50.5}
\end{align}
Together (\ref{e50.1}) and (\ref{e50.5}), we get $u\in L^{(p-1)a,\infty}(\O)$.

\end{proof}

\begin{clm}\label{c1}
    Suppose $\O\subset\rz^n$ is a bounded open set and $0<t<x$, then for every $v\in L^{x,\infty}(\O)$, we have
     $$
     \|v\|_{L^t(\O)}\le C\H^n(\O)^\tau\,\|v\|_{L^{x,\infty}(\O)} 
     $$
     where $\tau=\frac{1}{t}-\frac{1}{x}$ and $C$ depends only on $x,t$.
\end{clm}
\begin{proof}
    Using $(\ref{tool1})$, we have
    \begin{align*}
      \|v\|_{L^t(\O)}^t=\int_0^\infty t s^{t-1}\H^{n}(\O\cap\{|v|>s\}) d s
    \end{align*}
    Note that for $y>0$
    \begin{align*}
        \int_0^y t s^{t-1}\H^{n}(\O\cap\{|v|&>s\}) d s\le y^t\,\H^{n}(\O)\\
        \int_y^\infty t s^{t-1}\H^{n}(\O\cap\{|v|>s\}) d s&\le t\int_y^\infty s^{t-x-1}\|v\|_{L^{x,\infty}(\O)}^x d s\\
        &\le \frac{t}{t-x}y^{t-x}\|v\|_{L^{x,\infty}(\O)}^x
    \end{align*}
We choose $$y=\frac{\|v\|_{L^{x,\infty}(\O)}}{(\H^{n}(\O)^\frac{1}{x})}$$
Then the claim \ref{c1} follows.
    
\end{proof}

\begin{proof}[Proof of Proposition \ref{p1} (\ref{e17.1})]
    Note that
    \begin{align}
        \|u_R\|_{L^{(p-1)a,\infty}(\O)}^{(p-1)a}=&\sup_{h>0} h^{(p-1)a}\H^n(\O\cap \{x:R^\frac{n-p}{p-1} u(R x)>h\})\non\\
        =&\sup_{h>0} h^{(p-1)a}R^{-n}\H^n(\O\cap \{y:u(y)>h R^{-\frac{n-p}{p-1}} \})\non\\
        =&\sup_{d>0} d^{(p-1)a}\H^n(\O\cap \{x:u(y)>d\})\non\\
        =&\|u\|_{L^{(p-1)a,\infty}(\O)}^{(p-1)a}\label{e51}
    \end{align}
    Also, we have
    \begin{align}
        \|u_R\|_{L^{(p-1)b,\infty}(\p\O)}^{(p-1)b}=&\sup_{h>0} h^{(p-1)b}\H^{n-1}(\p\O\cap \{x:R^\frac{n-p}{p-1} u(R x)>h\})\non\\
        =&\sup_{h>0} h^{(p-1)a}R^{-(n-1)}\H^{n-1}(\p\O\cap \{y:u(y)>h R^{-\frac{n-p}{p-1}} \})\non\\
        =&\sup_{d>0} d^{(p-1)b}\H^{n-1}(\p\O\cap \{x:u(y)>d\})\non\\
        =&\|u\|_{L^{(p-1)b,\infty}(\p\O)}^{(p-1)b}\label{e52}
    \end{align}
Hence, $\|u_R\|_{L^{(p-1)b,\infty}(\p\O)}+\|u_R\|_{L^{(p-1)a,\infty}(\O)}<\infty$.\\
Note that
\begin{align*}
    &R^{-\frac{n-p}{p(p-1)}}\|u_R\|_{L^{p_*}(U\cap\p\O)}=\|u\|_{L^{p_*}(\p\O\cap\{R<|x|<8 R\})}\ra 0\,\,\T{as}\,\,R\ra \infty,\\
    &R^{-\frac{n-p}{p(p-1)}}\|u_R\|_{L^{p^*}(U\cap\O)}=\|u\|_{L^{p^*}(\O\cap\{R<|x|<8 R\})}\ra 0\,\,\T{as}\,\,R\ra \infty.
\end{align*}
Using Lemma \ref{l2} with $\l=R^{-1}$, $f(x)=V(x)$ and some $t\in(0,(p-1)a)$, we get that there exists $R_2>1$ such that
\begin{equation}\label{e53}
    \|u_R\|_{L^\infty(W\cap \O)}\le C\|u_R\|_{L^t(W_1\cap \O)}\,\,
    ,\T{for}\,\,R\ge R_2.
\end{equation}
Using claim \ref{c1}, $(\ref{e51})$ and $(\ref{e53})$, we obtain
\begin{equation}\label{e54}
    \|u_R\|_{L^\infty(W\cap \O)}\le C\|u_R\|_{L^{(p-1)a,\infty}(W_1\cap \O)}\le C\,\,
    ,\T{for}\,\,R\ge R_2.
\end{equation}
Hence, by elliptic regularity theory for p-Laplacian type equations \cite{L88,L91}, we obtain
$\|u_R\|_{C^{1,\a}(\O\cap \{4\le|x|\le 5\})}\le C$ for $R\ge R_2$, which yields $(\ref{e17.1})$.
    
\end{proof}

\subsubsection{Asymptotic lower bounds on $u$}
Adapting the proof in \cite{Veto16}, we shall prove $(\ref{e17.2})$ in this subsection.\\
Refer to \cite[Section 4.4]{HL00}, we see that weak Harnack inequality follows from local boundedness and John-Nirenberg lemma, and we have already shown local boundedness (that is lemma \ref{l2}).  
\begin{lem}\label{l4}
	Suppose $u\in W^{1,p}\cap L^\infty(U\cap \O )$ is a positive solution to
    \begin{equation}\label{e56}
    \left\{
    \begin{array}{lr}
	\Delta_p u=0 & ,\T{in} \,\,\O\cap U\\
	|D u|^{p-2}\frac{\partial u}{\partial x_n}=-f(x) u^{p-1} & ,\T{on} \,\,\p\O\cap U
	\end{array}
 \right.     
\end{equation}
where $f\in L^{\infty}(U\cap\p\O)$, $|f(x)|\le \l$ in $U\cap\p\O$ and $\l\ge 0$.\\
Then there exist $\d,p_1>0$ with the following property: let $c_n\in(0,1),\l$ be such that 
\begin{align*}
    \T Vol (B_{c_n})\le 1,\,\,\T Area (B_{c_n}\cap \p\O)\le 1,\,\,\l\le\d,
\end{align*}
   then for any $0<\rho\le c_n$, $x_0\in\O\cap W$, we have 
   $$\inf_{B_{\rho/2}(x_0)\cap \O} u\ge c\|u\|_{L^{p_1}(B_{\rho}(x_0)\cap \O)},$$
   where $\d,\,p_1$ depends only on $n,p$; and $c$ depends only on $n,p,\rho$.
\end{lem}
\begin{proof}
    Here we adapt some ideas which comes from \cite[Theorem 4.15]{HL00} and calculate like lemma \ref{l2}.\\
    Set $w=u^{-1}$. From $(\ref{e56})$ we have
    \begin{equation}
   \int_\O |\n u|^{p-2}\n u\cdot\n \vp d x=\int_{\p\O}f(x)u^{p-1} \vp d\s,\,\,\,\forall \vp\in C_c^\infty(U).
    \label{e57}  
   \end{equation}
   \textbf{Step 1:}\textit{ we prove that $\D\inf_{B_{\rho/2}(x_0)\cap \O} u\ge c\bigg(\|u^{-1}\|_{L^t(B_{\rho}(x_0)\cap \O)}\bigg)^{-1}$ for any $t>0$}.\\
Let $\phi\in C_c^\infty(U)$ and use $\vp=w^{2(p-1)}\phi$ as a test function in $(\ref{e57})$, we obtain
\begin{align}
   \int_\O |\n w|^{p-2}\n w\cdot\n \phi d x=&-\int_{\p\O}f(x)w^{p-1} \phi d\s-2(p-1)\int_\O |\n w|^{p} w^{-1} \phi d x\non\\
      \le& -\int_{\p\O}f(x)w^{p-1} \phi d\s.
   \label{e58}
\end{align}
As in the proof in lemma \ref{l1}, the proof is based on a Moser iteration argument. Let $\eta\in C_c^\infty(U),\,\b\ge 1$ and use $\phi=\eta^p w^\b$ as a test function in $(\ref{e58})$, then we get
\begin{align}
     \b\int_\O |\n w|^p\,\eta^p w^{\b-1} d x\le&-p\int_\O |\n w|^{p-2}\n w \cdot\n\eta\,\eta^{p-1} w^\b dx\non\\&-\int_{\p\O} f(x)\eta^p  w^{\b+p-1} d\s.\non
     \end{align}
     By H\"{o}lder's inequality, we get
     \begin{align}
     \b\int_\O |\n w|^p\,\eta^p w^{\b-1} d x&\le C_1\bigg(\b^{-(p-1)}\int_\O |\n \eta|^p w^{\b+p-1} dx\non\\&+\int_{\p\O} |f(x)|\eta^p  w^{\b+p-1} d\s\bigg)\non\\
     \RA\int_\O |\n(\eta  w^{\frac{\b+p-1}{p}})|^p d x&\le C_1\bigg(\int_\O |\n \eta|^p w^{\b+p-1} dx\non\\&+\l\b^p\int_{\p\O} \eta^p  w^{\b+p-1} d\s\bigg),
     \label{e59}
\end{align}
where $C_1$ always depend only on $p$.\\
It follows from H\"{o}lder's inequality and Sobolev trace inequality that
\begin{align}  
\l\int_{\p\O}\eta^p  w^{\b+p-1} d\s\le
&C_{n,p}\,\d\,\|\eta w^{\frac{\b+p-1}{p}}\|_{L^{p_*}(\p\O)}^p\non\\
\le& C_{n,p} \,\d\int_\O |\n(\eta w^{\frac{\b+p-1}{p}})|^p d x.\non
\end{align}
First, we fix $\b=p_*-p+1$ and choose $\d>0$ small enough which depends only on $n,p$. Then, we get
\begin{align}
    \int_\O |\n(\eta w^{\frac{\b+p-1}{p}})|^p d x\le&
    C_2\int_\O |\n \eta|^p w^{\b+p-1} dx,\non
\end{align}
where $C_2$ always depends only on $n,p$.\\
Hence, thanks to the Sobolev trace inequality and Observation \ref{o2}, we obtain 
\begin{align}
    \|\eta^p w^{p_*}\|_{L^a(\O)}+\|\eta^p w^{p_*}\|_{L^b(\p\O)}\le
    C_2\int_\O |\n \eta|^p w^{p_*} dx.\non
\end{align}
Second, let $\b$ vary and $\operatorname{supp}\eta\subset W_1$, we come back to $(\ref{e59})$.
Also, using the Sobolev trace inequality and Observation \ref{o2}, we obtain 
\begin{align}
    \|\eta^p w^{\b+p-1}\|_{L^a(\O)}+\|\eta^p w^{\b+p-1}\|_{L^b(\p\O)}&\le
    C_2\bigg(\int_\O |\n \eta|^p w^{\b+p-1} dx\non\\
    &+\b^p\int_{\p\O} \eta^p  w^{\b+p-1} d\s\bigg),\non
\end{align}
Since $\rho\le c_n,\,\T Vol (B_{c_n})\le 1,\,\T Area (B_{c_n}\cap \p\O)\le 1$, here we choose $\eta$ such that $\eta=1$ in $B_{r}(x_0)$, $\eta=0$ outside $B_{R}(x_0)$ and $|\n \eta|\le C_n/(R-r)$, where $0<r<R<\rho$. Then by H\"{o}lder's inequality we have
\begin{align}
    \|w^{\b+p-1}\|_{L^b(\O\cap B_r(x_0))}+\|w^{\b+p-1}&\|_{L^b(\p\O\cap B_r(x_0))}\le
    \non\\C_3\bigg(\frac{1}{(R-r)^p}\| w^{\b+p-1}\|_{L^1(\O\cap B_R(x_0))}
+&\b^p\| w^{\b+p-1}\|_{L^1(\p\O\cap B_R(x_0))}\bigg).\label{e60}
\end{align}
where $C_3$ always depend only on $n,\,p,\,\rho$.\\
As in the proof of lemma \ref{l1}, the classical Moser iteration argument yields that
$$\|w\|_{L^\infty(B_{\rho/2}(x_0)\cap \O)}\le C\|w\|_{L^t(B_{\rho}(x_0)\cap \O)},$$
that is,
$$\inf_{B_{\rho/2}(x_0)\cap \O} u\ge c\bigg(\|u^{-1}\|_{L^t(B_{\rho}(x_0)\cap \O)}\bigg)^{-1}.$$
 Set $$v=\log u-\frac{\int_{\O\cap B_{\rho}(x_0)}\log u \,d x}{|\O\cap B_{\rho}(x_0)|}.$$
\textbf{Step 2:}\textit{ we prove that $v\in BMO$}.\\
Using $\vp=u^{-(p-1)}\eta^p$ as a test function in $(\ref{e57})$, it follows from H\"{o}lder's inequality that
\begin{align}
   (p-1)\int_\O |\n v|^{p} \eta^p d x=&-\int_{\p\O}f(x) \eta^p d\s-p\int_\O |\n v|^{p-2}\n v\cdot\n \eta \eta^{p-1} d x\non\\
     \RA \int_\O |\n v|^{p} \eta^p d x\le& C_1\bigg(\l\int_{\p\O}\eta^p d\s+\int_\O |\n \eta|^{p}d x\bigg).
   \label{e61}
\end{align}
For any ball $B_{r}(y)\subset B_{\rho}(x_0)$, we choose $\eta$ such that $\eta=1$ in $B_{r}(y)$, $\eta=0$ outside $B_{2 r}(y)$ and $|\n \eta|\le C_n/r$. Then by Poincare inequality and $(\ref{e61})$, we have
\begin{align}
   \int_{\O\cap B_r(y)} | v-(v)_r|^{p} d x&\le C_2 r^p\int_{\O\cap B_r(y)} |\n v|^{p} d x\non\\
   &\le C_2 r^p(\l r^{n-1}+r^{n-p})\non\\
   &\le C_2 r^n,\non
\end{align}
where $(v)_r:=\frac{\int_{\O\cap B_{r}(y)}v d x}{|\O\cap B_{r}(y)|}$.\\
Hence, by H\"{o}lder's inequality, we get
\begin{align}
r^{-n}\int_{\O\cap B_r(y)} | v-(v)_r| d x\le C_n r^{-\frac{n}{p}}\bigg(\int_{\O\cap B_r(y)} | v-(v)_r|^{p} d x\bigg)^\frac{1}{p}
\le  C_2,\non
\end{align}
which implies that $v\in \T{BMO}$.\\
\textbf{Step 3:}\textit{ for any $p_1>0$, we have
   $\inf_{B_{\rho/2}(x_0)\cap \O} u\ge c\|u\|_{L^{p_1}(B_{\rho}(x_0)\cap \O)}.$}\\
Note that John-Nirenberg lemma (\cite[Theorem 3.5]{HL00}) actually holds for cubes, so this lemma can be applied to $\O\cap B_{\rho}(x_0)$. \\
That is, if for any ball $B_{r}(y)\subset B_{\rho}(x_0)$, we have
\begin{align}
\int_{\O\cap B_r(y)} | v-(v)_r| d x\le M r^{n},\non
\end{align}
    then we have
    $$
    \int_{\O\cap B_\rho(x_0)} e^{\frac{p_0}{M}|v|} d x\le C
    $$
    for some $p_0,C$ depend only on $n$.\\
    Therefore, there exists $p_1$ depending only on $n,p$ such that
    \begin{align}
        \int_{\O\cap B_\rho(x_0)} e^{p_1|v|} d x&\le C_n\non\\
        \RA \|u^{-1}\|_{L^{p_1}(B_{\rho}(x_0)\cap \O)}\,&\|u\|_{L^{p_1}(B_{\rho}(x_0)\cap \O)}\le C_2.\label{e62}
    \end{align}
Together $(\ref{e62})$ with step 1, we obtain the result.

\end{proof}

\begin{cor}\label{cor1}
    Suppose $u$ is a solution of $(\ref{e1})$. Then there exists $R_3>1$ such that
   $$\sup_{W\cap \O} u_R\le C\inf_{W\cap \O} u_R\quad\T{for}\,\,R\ge R_3,$$
   where $C$ depends only on $n,p$.
\end{cor}
\begin{proof}
    In section 2.2.1, we have shown 
\begin{align*}
    \|u_R\|_{L^\infty(W\cap \O)}\le  C\,\,
    ,\T{for}\,\,R\ge R_2\\
    \RA |V(x)|\le C R^{-1}\ra 0\,\,
    ,\T{as}\,\,R\ra\infty.
\end{align*}
So, by lemma \ref{l4}, there exists $R_3>1,\,p_1>0$ such that for any $0<\rho\le c_n$, $x_0\in\O\cap W$, we have 
   $$\inf_{B_{\rho/2}(x_0)\cap \O} u_R\ge c\|u_R\|_{L^{p_1}(B_{\rho}(x_0)\cap \O)}\quad\T{for}\,\,R\ge R_3,$$
   where $c$ depends only on $n,p,\rho$; and $p_1$ depends only on $n,p$.\\
Using lemma \ref{l2} with $\l=R^{-1},\,f(x)=V(x),\,t=p_1$, we obtain that there exists $R_3>1$ such that for any $0<\rho\le c_n$, $x_0\in\O\cap W$, we have
   $$\sup_{B_{\rho/2}(x_0)\cap \O} u_R\le C\inf_{B_{\rho/2}(x_0)\cap \O} u_R\quad\T{for}\,\,R\ge R_3,$$
   where $C$ depends only on $n,p,\rho$.\\
   Then, choose $\rho=c_n$ and a classical covering argument ends the proof of this corollary.
    
\end{proof}

\begin{proof}[Proof of Proposition \ref{p1} (\ref{e17.2})]
To prove $(\ref{e17.2})$, it suffices to prove there exists $R_1>1$ such that if $R>R_1$, then
$$ u_R(x)\ge C_0\,\,\T{in}\,\,\O\cap W.$$
Following the proof of \cite{Veto16}, we assume by
contradiction that $(\ref{e17.2})$
does not hold true. Then, it follows from corollary \ref{cor1} that there exists a sequence $(R_k)_{k\in \nz}$ such that 
\begin{align}
    R_k\ra\infty\quad \T{and}\quad \sup_{W\cap \O} u_{R_k}\ra 0\,\,\,\T{as}\,\,\,k\ra\infty.\label{e63}
\end{align}
Since $u_R$ satisfy $(\ref{e20})$,  by elliptic regularity theory for p-Laplacian type equations \cite{L88,L91}, we obtain
$\|u_{R_k}\|_{C^{1,\a}(\O\cap \{4\le |x|\le 5\})}\le C$, where $C$ is independent of $k$. It then follows from
$(\ref{e63})$ that up to a  subsequence $u_{R_k}\ra 0$ in $C^{1}(\O\cap \{4\le |x|\le 5\})$. Let $\eta\in C_c^\infty(\rz^n)$ be a cut-off function such that $\eta=1$ in $B_{4}$, $\eta=0$ outside $B_{5}$ and $|\n \eta|\le C_n$. Denote by $\eta_k (x):=\eta(x/R_k)$. By testing $(\ref{e20})$
with $\eta$ and using H\"{o}lder’s inequality, we obtain
\begin{align}
    \int_{\O} |\n u_{R_k}|^{p-2}\n u_{R_k}\cdot\n\eta \,d x&=
    \int_{\p\O} R_k^{-1} u_{R_k}^q \eta \,d\s=\int_{\p\O} u^q \eta_k \,d\s\non\\
    \RA \int_{\p\O} u^q \eta_k \,d\s\le
    C_n\|\n u_{R_k}&\|_{L^\infty(\O\cap \{4\le |x|\le 5\})}^{p-1}\ra 0\,\,\,\T{as}\,\,\,k\ra \infty.
    \label{e64}
\end{align}
It follows from $(\ref{e64})$ that $u\equiv 0$ on $\p\O$. Then by testing $(\ref{e20})$
with $\eta^p u_{R_k}$ and using $\|u_{R_k}\|_{C^{1,\a}(\O\cap \{4\le |x|\le 5\})}\le C$, we obtain
\begin{align}
    \int_{\O} |\n u_{R_k}|^p\eta^p d x=&
    -p\int_{\O} |\n u_{R_k}|^{p-2}\n u_{R_k}\cdot\n\eta\eta^{p-1}u_{R_k} d x\non\\
    \RA \int_{\O} &|\n u_{R_k}|^p\eta^p d x\le C.
    \label{e65}
\end{align}
It follows from $(\ref{e19})$ and $(\ref{e65})$ that
\begin{align}
    \int_{\O} |\n u|^p\eta_k^p d x\le C R_k^\frac{p-n}{p-1}\ra 0\,\,\,\T{as}\,\,\,k\ra \infty.\label{e66}
\end{align}
It follows from $(\ref{e66})$ that $u\equiv 0$ on $\O$, which is in contradiction with $(\ref{e1})$. This ends the proof.

\end{proof}

\subsubsection{Asymptotic limits of $u$ and $\n u$}
The proof of $(\ref{e17.3})$ closely follows \cite[Section 4]{Veto16}.\\
Define
$$
\G_R(u):=\operatorname*{min}_{x\in S^{n-1}_+}R^{\frac{n-p}{p-1}}u(R x)
$$
where $S^{n-1}_+$ is the upper half of unit sphere. 

\begin{proof}[Proof of Proposition \ref{p1} (\ref{e17.3})]
Firstly, we prove the following claim.
\begin{clm}\label{c2}
    There exists $\a>0$ such that
    $\D\lim_{R\ra\infty}\G_R(u)=\a$.
\end{clm}
By $(\ref{e17.2})$, we get that $\D\a:=\liminf_{R\ra\infty}\,\G_R(u)>0$.
Now we assume by
contradiction that Claim \ref{c2}
does not hold true. Then, we have $\D\limsup_{R\ra\infty}\,\G_R(u)>\a$, which implies there exists $R_2>R_1>1$ such that
$$
0<\b:=\min_{R_1\le R\le R_2}\G_R(u)<\min\{\G_{R_1}(u),\,\G_{R_2}(u)\}.
$$
Set $A=\O\cap B_{R_2}\bs B_{R_1}$ and $w(x):=\b|x|^{-\frac{n-p}{p-1}}$. Then, we have
\begin{equation}\label{ec2}
    \left\{
    \begin{array}{lr}
	\Delta_p u\le \Delta_p w=0 & ,\T{in} \,\,A\\
	|D u|^{p-2}\frac{\partial u}{\partial t}<0=|D w|^{p-2}\frac{\partial w}{\partial t} & ,\T{on} \,\,\p\O\cap \p A\\
 \D\min_A(u-w)=0<\min_{\O\cap \p A}(u-w)&
	\end{array}
 \right.     
\end{equation}
We then obtain that $(\ref{ec2})$ contradicts the strict comparison principle
of \cite{S70}. This ends the proof of Claim \ref{c2}.\\
Second, we assume by
contradiction that $(\ref{e17.3})$
does not hold true. Then, 
there exists $(R_k)_{k\in\nz}$ such that $R_k\ra\infty$ as $k\ra\infty$ and
\begin{equation}
\D\limsup_{k\ra\infty}\sup_{x\in S^{n-1}_+}\bigg(\big|u_{R_k}(x)-\a\big|+\big|\n u_{R_k}(x)+\a\frac{n-p}{p-1}x\big|\bigg)>0
\label{ec3}
\end{equation}
Since we have shown in Section 2.2.1 that for every compact set $B\subset\ov\O$, there exists
$\a_B\in(0,1)$ such that $(u_{R_k})_{k\in\nz}$
is bounded in $C^{1,\a_B}(B)$ and so there
exists a subsequence of $(u_{R_k})_{k\in\nz}$ which converges in $C^1_{loc}(\ov \O)$ to some
function $u_0\in C^1_{loc}(\ov\O)$. By passing to the limit as $k\ra\infty$ into $(\ref{e20})$, we obtain that $u_0$ satisfies the equation
\begin{equation}
    \left\{
    \begin{array}{lr}
	\Delta_p u_0=0 & , \T{in} \,\,\O\\
	|D u_0|^{p-2}\frac{\partial u_0}{\partial t}=0 & , \T{on} \,\,\p\O
	\end{array}
 \right.     \non
\end{equation}
By using Claim \ref{c2} and observing that $\G_r(u_{R_k})=\G_{r R_k}(u)$,  we obtain
$$
\G_r(u_0)=\lim_{k\ra\infty}\G_r(u_{R_k})=\a\quad\forall r>0.
$$
Also, using the strict comparison principle
of \cite{S70}, we get $u_0(x)=\a|x|^{-\frac{n-p}{p-1}}$ in $\O$, which contradicts $(\ref{ec3})$. This ends the proof of $(\ref{e17.3})$.

\end{proof}

\subsection{Asymptotic integral estimate on second order derivatives}
By using a Caccioppoli-type
inequality like \cite{CFR20}, in this subsection we will prove the following proposition.
\begin{prop}\label{p2}
   Let $u$ be a solution to $(\ref{e1})$. Then $|\n u|^{p-2}\n u\in W^{1,2}_{loc}(\ov\O)$, and for any
$\g\in\rz$ the following asymptotic estimate holds:
\begin{align}
    &\int_{B_R\cap\O}|\n(|\n u|^{p-2}\n u)|^2 u^\g d x\le C(1+R^{-n-\g\frac{n-p}{p-1}})\qquad \forall\, R>1,\non\\
    &\int_{\O\cap(B_{2 R}\bs B_R)}|\n(|\n u|^{p-2}\n u)|^2 u^\g d x\le C R^{-n-\g\frac{n-p}{p-1}}\qquad \forall\, R>1,\non
\end{align}
   where $C$ is a positive constant independent of $R$.
\end{prop}

\begin{proof}

This proposition is obtained by using a Caccioppoli-type inequality. We argue by
approximation, following the approach in \cite{CFR20,CM18}.\\
We set
\begin{align*}
a_i(x):=|x|^{p-2}x_i,&\,\,H(x):=|x|^p\quad\T{for}\,\,x\in\rz^n,\\
a^k_i(x):=(a_i*\phi_k)(x),&\,\,H_k(x):=(H*\phi_k)(x)\quad\T{for}\,\,x\in\rz^n,
\end{align*}
where $\{\phi_k\}$ is a family of radially symmetric smooth mollifiers.
Standard properties 
of convolution and the fact $a(x)$ is continuous imply $a^k\ra a$ uniformly on compact subset of $\rz^n$. 
From \cite[Lemma 2.4]{FF97} and \cite{CFR20}, we have that $a^k$
satisfies the first two condition in $(\ref{e5})$ with $s$ replaced by $s_k$
, where
$s_k\ra 0$ as $k\ra \infty$.\\
For $k$ fixed, we let $u_k\in D^{1,p}(\O)$ be a solution of 
\begin{equation}\label{e67}
    \left\{
    \begin{array}{lr}
	\div (a^k(\n u_k))=0 & \T{in} \,\,\O,\\
	a^k(\n u_k)\cdot \nu=u^q & \T{on}\,\, \p \O,
	\end{array}
 \right.
\end{equation}
where $\nu$ is the unit outward normal of $\p\O$.

Note that $u_k$ can be constructed as follows. First, we consider the minimizer $u_k^R$ of the variation problem
\begin{align}
    \min_{v}\bigg\{\int_{\O\cap B_R}\frac{1}{p}H_k(\n v)\, d x-\int_{\p\O\cap B_R}u^q v\, d\s: v\in C_0^\infty(B_R)\bigg\},\label{e68}
\end{align}
then $u_k^R$ is the solution of
\begin{equation}\label{e69}
    \left\{
    \begin{array}{lr}
	\div (a^k(\n u_k^R))=0 & \T{in} \,\,\O\cap B_R,\\
	a^k(\n u_k^R)\cdot \nu=u^q & \T{on}\,\, \p\O\cap B_R,\\
 u_k^R=0 & \T{on}\,\, \O\cap \p B_R.
	\end{array}
 \right.
\end{equation}
We extend $u_k^R$ to be $0$ in $\O\cap B_R^c$. Then, by testing $(\ref{e69})$
with $u_k^R$, using $\|u\|_{L^{p_*}(\p\O)}\le C$ and Sobolev trace inequality, we obtain
\begin{align}
   C_{n,p}( \|u_k^R\|_{L^{p_*}(\p\O)}^p-1)\le\int_{\O} a^k(\n u_k^R)\cdot\n u_k^R \,d x&=
    \int_{\p\O} u_k^R u^q\,d\s\le C\|u_k^R\|_{L^{p_*}(\p\O)}\non\\
    \RA \|u_k^R\|_{L^{p_*}(\p\O)}\le C,\,\,&\|\n u_k^R\|_{L^{p}(\O)}\le C.
    \label{e70}
\end{align}
Also, by Observation \ref{o2} and $(\ref{e70})$, we have $\|u_k^R\|_{L^{p^*}(\O)}\le C$, where $C$ is independent of $k$ and $R$. Hence, there exist $ (u_k)_{k\in\nz}$ which is bounded in $D^{1,p}(\O)$, such that as $R\ra \infty$, up to a subsequence, we have
\begin{align}
    u_k^R\ru  u_k\,\,\T{in}\,\,L^{p^*}(\O),\,u_k^R\ru u_k\,\,\T{in}\,\,L^{p_*}(\p\O)\,\,,\,\,a^k(\n u_k^R)\ru a^k(\n u_k)\,\,\T{in}\,\,L^{\frac{p}{p-1}}(\O).\label{e71}
\end{align}
Testing $(\ref{e69})$
with $\eta\in C_c^\infty(\rz^n)$ and passing to the limit as $R\ra\infty$, it follows from $(\ref{e71})$ that
$u_k\in D^{1,p}(\O)$ is a solution of $(\ref{e67})$.\\
Since $\|u\|_{L^\infty(\O)}\le C$, it follows from Lemma \ref{l1} and elliptic regularity theory for p-Laplacian type equations \cite{L88,L91} that $u_k\in C^{1,\a}_{loc}(\ov \O)$, uniformly in $k$. So $(u_k)_{k\in\nz}$
converges up to a subsequence in
$C^1_{loc}(\ov \O)$ to some function $\ov u\in C^{1,\a}_{loc}(\ov \O)\cap D^{1,p}(\O)$, which is 
a solution of 
\begin{equation}\label{e72}
    \left\{
    \begin{array}{lr}
	\div (a(\n \ov u))=0 & \T{in} \,\,\O,\\
	a(\n \ov u)\cdot \nu=u^q & \T{on}\,\, \p \O.
	\end{array}
 \right.
\end{equation}
Testing $(\ref{e1})$ and $(\ref{e72})$ with $(\ov u-u)$, we obtain
$$
\int_\O (a(\n \ov u)-a(\n u))\cdot\n(\ov u-u) d x=0,
$$
which implies $\ov u-u\equiv C$. Since $u$ and $\ov u\in L^{p^*}(\O)$, we get $\ov u=u$.\\
Note that $(\ref{e67})$ is uniformly elliptic and implies $u_k\in C^\infty(\O)\cap W^{2,2}_{loc}(\ov \O)$. Since $u_k\in C^{1,\a}_{loc}(\ov \O)$, we have $a^k(\n u_k)\in W^{1,2}_{loc}(\ov \O)$. In addition, since $\O$ is smooth, $u_k\in C^{2,\a}_{loc}(\ov\O)$.

Given $R>1$ large, let $U$ be a $C^2$ domain such that
$\O\cap B_{4R}\subset U\subset\O\cap B_{5R}$. 
Note that, since $u$ is uniformly positive inside $ U$ (see Proposition \ref{p1}), for $k$ large enough
(depending on $R$) $u_k$ is also uniformly positive inside $ U$. We shall always assume that $k$ are sufficiently large so that this positivity property holds below. 

To simplify the notation, we shall drop the dependency of $a$ on k and we write $a$ instead of $a^k$ respectively.\\
Let $\psi\in C_c^{0,1}(B_{4R})$ and test $(\ref{e67})$ with $\psi$, we have
\begin{align}
   \int_{\O} a(\n u_k)\cdot\n \psi \,d x=
    \int_{\p\O} u^q \psi\,d\s.
    \label{e73}
\end{align}
For $\d>0$ small, we define the set
$$
 U_\d:=\{x\in U:\operatorname{dist}(x,\p U)>\d\}.
$$
It follows from \cite{GT77} that for small $\d>0$, we have $d(x):=\operatorname{dist}(x,\p U)\in C^2(U\bs U_{2\d})$ and every point $x\in U\bs U_{2\d}$ can be uniquely written as $x=y-|x-y|\nu(y)$, where $y(x)\in\p U$ and $\nu(y)$ is the unit outward normal to $\p U$ at $y$.
\\
Set 
$$
\g(t)=\begin{cases}
    0 & t\in[0,\d],\\
    \frac{t-\d}{\d} & t\in [\d,2\d),\\
    1 & t\in [2\d,\infty),
\end{cases}
\quad\T{and}\quad\xi_\d(x)=\g(d(x))\,\,\T{in}\,\, U.
$$
Then we have $\xi_\d\in C_c^{0,1}( U)$ satisfy
$$
\xi_\d=1\,\,\T{in}\,\, U_{2\d},\,\xi_\d=0\,\,\T{in}\,\, U\bs U_\d\,\,\T{and}\,\,\n\xi_\d(x)=-\frac{1}{\d}\nu(y(x))\,\,\T{inside}\,\, U_\d\bs U_{2\d}.
$$
Let $\vp\in C^2( U)\cap C^{1}(\ov{ U})$ and $\operatorname{supp}(\vp)\Subset B_{4R}$. Using $\psi=\p_m\vp\,\xi_\d$ in $(\ref{e73})$ with $1\le m\le n$ and integrating by parts, we obtain
\begin{align}
    \sum_i\bigg(-\int_{ U}\p_{i}\vp\,\p_m(a_i(\n u_k)\xi_\d) d x&+\int_{ U}a_i(\n u_k)\p_{m}\vp\,\p_i \xi_\d d x\bigg)=0\non\\
    \RA \sum_i\int_{ U}\p_{i}\vp\p_m a_i(\n u_k)\xi_\d d x=&\sum_i\bigg(-\int_{ U}\p_{i}\vp\,a_i(\n u_k)\p_m\xi_\d d x\non\\
    &+\int_{ U}a_i(\n u_k)\p_{m}\vp\,\p_i \xi_\d d x\bigg).\label{e74}
\end{align}
Since $\vp\in C^{1}(\ov{ U})$ and $\operatorname{supp}(\vp)\Subset B_{4R}$, we can pass to the limit as $\d\ra 0$ into $(\ref{e74})$ and get
\begin{align}
    \sum_i\int_{ U}\p_{i}\vp\,\p_m a_i(\n u_k) d x=&\sum_i\bigg(\int_{ \p \O}\p_{i}\vp\,a_i(\n u_k)\nu^m d \s\non\\
    &-\int_{\p\O}a_i(\n u_k)\p_{m}\vp\,\nu^i d \s\bigg),\label{e75}
\end{align}
where $\nu=(\nu^1,\dots,\nu^n)$ is then unit outward normal to $\p\O$. Since $\O=\rz^n_+$, we have $\nu=(0,\dots,0,-1)$ on $\p\O$.\\
Let $\eta\in C_c^\infty(B_{4R})$. We choose $\vp=a_m(\n u_k) u^\g \eta^2$ and divide $m$ into two cases.\\
\textbf{Case 1:} $m\neq n$\\
Then, it follows from $(\ref{e67})$ that $(\ref{e75})$ becomes
\begin{align}
    \sum_i\int_{U}\p_m a_i(\n u_k)\,\p_{i}(a_m(\n u_k) u^\g &\eta^2)\, d x=-\sum_i\int_{\p\O}u^q\p_{m}\vp\, d \s\non\\
=&\sum_i\int_{\p\O}\p_m(u^q)a_m(\n u_k) u^\g \eta^2\, d \s
    \label{e76}
\end{align}

\noindent\textbf{Case 2:} $m=n$\\
Note that $\vp=-u^{q+\g}\eta^2$ on $\p\O$ in this case.
Then, $(\ref{e75})$ becomes
\begin{align}
    \sum_i\int_{ U}\p_{i}\vp\,\p_m a_i(\n u_k) d x=&\sum_{i=1}^n\int_{ \p \O}\p_{i}\vp\,a_i(\n u_k)\nu^n d \s-\int_{\p\O}a_n(\n u_k)\p_{n}\vp\,\nu^n d \s\non\\
    =-\sum_{i=1}^{n-1}&\int_{\p \O}\p_{i}\vp\,a_i(\n u_k) d \s=\sum_{i=1}^{n-1}\int_{ \p \O}\p_{i}(u^{q+\g}\eta^2)\,a_i(\n u_k) d \s
    \label{e77}
\end{align}
Combining $(\ref{e76})$ and $(\ref{e77})$, we obtain
\begin{align}
\sum_{i,m=1}^n\bigg|\int_{U}\p_m a_i(\n u_k)\,\p_{i}(a_m(\n u_k) u^\g \eta^2)\, d x\bigg|&\le C\int_{\p\O}|a(\n u_k)|\,|\n u| u^{\g+q-1} \eta^2\, d \s\non\\
&+C \int_{\p\O}|a(\n u_k)|\,|\n \eta| u^{\g+q} \eta\, d \s,
    \label{e78}
\end{align}
where $C$ only depends on $n,p,\g$.\\
Using \cite[Lemma 4.5]{AKM18} and
arguing as in \cite{AKM18}, we obtain
\begin{align}
\int_{U}|\n a(\n u_k)|^2 u^\g \eta^2\, d x&\le C\bigg(\int_{U}|a(\n u_k)|^2 |\n (u^\frac{\g}{2} \eta)|^2\, d x\non\\+\int_{\p\O}|a(\n u_k)|&|\n u| u^{\g+q-1} \eta^2\, d \s
+\int_{\p\O}|a(\n u_k)|\,|\n \eta| u^{\g+q} \eta\, d \s\bigg).
    \label{e79}
\end{align}
Since $(u_k)_{k\in\nz}$ is bounded in $C^{1,\a}_{loc}(\ov \O)$, taking $\g=0$, it follows from $(\ref{e79})$ that $a^k(\n u_k)\in W^{1,2}_{loc}(\ov \O)$ and $\{a^k(\n u_k)\}_{k\in\nz}$ is uniformly bounded in $W^{1,2}_{loc}(\ov \O)$. \\
Since $(u_k)_{k\in\nz}$
converges up to a subsequence in
$C^1_{loc}(\ov \O)$ to $u$ and $a^k\ra a$ uniformly on compact subset of $\rz^n$, we get $a^k(\n u_k)\ra a(\n u)$ in $C^0_{loc}(\ov \O)$.
Therefore, as $k\ra\infty$, up to a subsequence, we have $a^k(\n u_k)\ru a(\n u)$ in 
$W^{1,2}_{loc}(\ov\O)$ and $a(\n u)\in 
 W^{1,2}_{loc}(\ov\O)$.
So, up to a subsequence, by passing to the limit as $k\ra\infty$ into $(\ref{e79})$ , we obtain
\begin{align}
\int_{U}|\n a(\n u)|^2 u^\g \eta^2\, d x&\le C\bigg(\int_{U}|a(\n u)|^2 |\n (u^\frac{\g}{2} \eta)|^2\, d x\non\\+\int_{\p\O}|a(\n u)|&|\n u| u^{\g+q-1} \eta^2\, d \s
+\int_{\p\O}|a(\n u)|\,|\n \eta| u^{\g+q} \eta\, d \s\bigg).
    \label{e80}
\end{align}
First, let $\eta$ be a non-negative cut-off function such that $\eta=1$ in $B_R$, $\eta=0$ outside $B_{2R}$ and $|\n \eta|\le \frac{C_n}{R}$. It follows from Proposition \ref{p1} and $(\ref{e80})$ that
\begin{align}
    \int_{B_R\cap\O}|\n(a(\n u))|^2 u^\g d x\le C(1+R^{-n-\g\frac{n-p}{p-1}})\qquad \forall\, R>1.\non
\end{align}
Second, let $\eta$ be a non-negative cut-off function such that $\eta=1$ in $B_{2 R}\bs B_R$, $\eta=0$ outside $B_{3R}\bs B_\frac{R}{2}$ and $|\n \eta|\le \frac{C_n}{R}$. Using Proposition \ref{p1} and $(\ref{e80})$, we get
\begin{align}
    \int_{(B_{2 R}\bs B_R)\cap\O}|\n(a(\n u))|^2 u^\g d x\le C R^{-n-\g\frac{n-p}{p-1}}\qquad \forall\, R>1,\non
\end{align}
where $C$ is independent of $R$.

\end{proof}

\section{Proof of Theorem \ref{t1}}
To simplify the computation, we consider the auxiliary function
\begin{equation}
    v=\frac{n-p}{p}u^{-\frac{p}{n-p}},\non
\end{equation}
where $u$ is a solution of $(\ref{e1})$. A straightforward computation shows that $v > 0$ satisfies the
following equation
\begin{equation}\label{e81}
    \left\{
    \begin{array}{lr}
	\Delta_p v=\frac{n(p-1)}{p}\frac{|\n v|^p}{v} & \T{in}\quad \O\\
	|\n v|^{p-2}\frac{\partial v}{\partial t}=1 & \T{on}\quad \p\O
	\end{array}
 \right.     
\end{equation}
In particular, $v\in C_{loc}^{1,\a}(\ov\O)$ for some $\a\in(0,1)$, and it follows from
Proposition \ref{p1} that there exist constants $C_0,C_1>0$ such that
\begin{equation}
    C_0\le |x|^{-\frac{p}{p-1}}v(x)\le C_1,\,\,|x|^{-\frac{1}{p-1}}|\n v(x)|\le C_1\label{e82}
\end{equation}
for $|x|$ sufficiently large. Also, Proposition \ref{p2} becomes that $|\n v|^{p-2}\n v\in W^{1,2}_{loc}(\ov\O)$ and
\begin{align}
    \int_{\O\cap(B_{2 R}\bs B_R)}|\n(|\n v|^{p-2}\n v)|^2 v^\s d x\le C R^{n+\s\frac{p}{p-1}}\qquad \forall\, R>1,\label{e83}
\end{align}
   where $C$ is a positive constant independent of $R$.

\subsection{A differential identity}
In this subsection, by using Bochner's skill and a Pohozaev identity like \cite{GL23}, we show
that $u$ satisfies the following differential identity, which implies Theorem \ref{t1}.
\begin{prop}\label{p3}
     Let $u$ be a solution to $(\ref{e1})$. Then we have
     $$\p_j (|\n u|^{p-2}u_i)=q\bigg(\frac{|\n u|^{p-2}u_i u_j}{u}-\frac{|\n u|^p}{n u}\d_{i j}\bigg)\quad\T{in}\,\,\,L^2_{loc}(\ov \O).$$
     In particular, we have $|\n u|^{p-2}\n u\in C^{1,\a}_{loc}(\ov\O)$.
\end{prop}
\begin{proof}
As in the proof of Proposition \ref{p2}, proof of integral inequality requires a regularization argument considering the solutions of the following approximating equations.\\
We set
\begin{align*}
a_i(x):=|x|^{p-2}x_i,&\,\,H(x):=|x|^p\quad\T{for}\,\,x\in\rz^n,\\
a^k_i(x):=(\frac{1}{k^2}+|x|^2)^{\frac{p-2}{2}}x_i,&\,\,H_k(x):=(\frac{1}{k^2}+|x|^2)^\frac{p}{2}\quad\T{for}\,\,x\in\rz^n,
\end{align*}
Straightforward computation implies that $a^k\ra a$ uniformly on compact subset of $\rz^n$ and $a^k$
satisfies the first two condition in $(\ref{e5})$ with $s$ replaced by $\frac{1}{k^2}$.\\
Arguing as the proof of Proposition \ref{p2}, we obtain $u_k\in D^{1,p}(\O)$ being a solution of 
\begin{equation}\label{e87}
    \left\{
    \begin{array}{lr}
	\div (a^k(\n u_k))=0 & \T{in} \,\,\O,\\
	a^k(\n u_k)\cdot \nu=u^q & \T{on}\,\, \p \O.
	\end{array}
 \right.
\end{equation}
Also, we have $u_k\in C^\infty(\O)\cap C^{2,\a}_{loc}(\ov\O)$ and $u_k\in C^{1,\a}_{loc}(\ov \O)$, uniformly in $k$. Hence, $(u_k)_{k\in\nz}$
converges up to a subsequence in
$C^1_{loc}(\ov \O)$ to $u$. Besides, the proof of Proposition \ref{p2} shows that as $k\ra\infty$, up to a subsequence, we have $a^k(\n u_k)\ru a(\n u)$ in 
$W^{1,2}_{loc}(\ov\O)$. In the following, we only consider the subsequence of $(u_k)_{k\in\nz}$ satisfying that $u_k\ra u$ in $C^1_{loc}(\ov \O)$ and $a^k(\n u_k)\ru a(\n u)$ in 
$W^{1,2}_{loc}(\ov\O)$.\\
Set 
\begin{align}
    &X=(X_i),\,X_i:=|\n u|^{p-2}u_i,\,X_{ij}:=\p_j X_i,\,L_{ij}:=\frac{X_i u_j}{u}-\frac{|\n u|^p}{n u}\d_{i j},\non\\
    &X^k_i:=a^k_i(\n u_k),\,X^k_{ij}:=\p_j X^k_i,\,L^k_{ij}:=\frac{X^k_i \p_j u_k}{u}-\frac{H_k(\n u_k)}{n u}\d_{i j}-D^k_{ij},
    \label{e88}
\end{align}
where 
$$
D^k_{ij}=\frac{p-2}{n u}(\frac{n-1}{p-1}+1)(\frac{1}{k^2}+|\n u_k|^2)^\frac{p-2}{2}\frac{\frac{1}{k^2}\p_i u_k \p_j u_k}{(\frac{1}{k^2}+|\n u_k|^2)}.
$$
Note that $L_{ij}$, $L^k_{ij}$ and $D^k_{ij}$ are symmetry, but $X_{ij}$ is not.\\
Since $ab\le \frac{(a+b)^2}{4}$, $u>0$ and $u_k\in C^{1,\a}_{loc}(\ov \O)$, uniformly in $k$, we can see $D^k\ra 0$ uniformly on compact subset of $\ov\O$.\\
Note that we can write the matrix $X^k_{ij}-q L^k_{ij}$ as $A (B+\frac{q}{n} C)$, where 
\begin{align}
    &A_{ij}=\d_{ij}+\frac{(p-2)\p_i u_k \p_j u_k}{(\frac{1}{k^2}+|\n u_k|^2)},\,\,B_{ij}=(\frac{1}{k^2}+|\n u_k|^2)^\frac{p-2}{2}\p_{ij}u_k\non\\
    &C_{ij}=\frac{(\frac{1}{k^2}+|\n u_k|^2)^\frac{p}{2}}{u}\bigg[\d_{ij}-(\frac{n-1}{p-1}+1)\frac{\p_i u_k \p_j u_k}{(\frac{1}{k^2}+|\n u_k|^2)}\bigg]\non
\end{align}
It follows from \cite[Lemma 4.5]{AKM18} that
\begin{equation}\label{e89}
    \D\sum_{i,j}(X^k_{ij}-q L^k_{ij})(X^k_{ji}-q L^k_{ji})\ge C_p\sum_{i,j}|X^k_{ij}-q L^k_{ij}|^2 \ge 0.
\end{equation}

To simplify the notation, we shall drop the dependency on k and we write $a$, $\tl u$, $X$, $X_i$, $X_{ij}$, $L_{ij}$, $D_{ij}$ instead of $a^k,\,u_k,\,X^k,\,X^k_i,\,X^k_{ij},\,L^k_{ij}$, $D^k_{ij}$ respectively.\\
First, using $(\ref{e87})$, that is $\sum_i X_{ii}=0$, we get the following Serrin-Zou type differential equality
\begin{align}
    &\sum_i\p_i\bigg(\sum_j u^{1-q} X_{ij}X_j-\frac{(n-1)q}{n} u^{-q}|X|^{\frac{p}{p-1}}X_i\bigg)\non\\
    &=\sum_{i,j}\bigg(u^{1-q} X_{iij}X_j+u^{1-q} X_{ij}X_{ji}
    +(1-q)u^{-q} X_{ij}X_j u_i\non\\
    &-\frac{(n-1) p}{n-p}u^{-q}|X|^{\frac{1}{p-1}}\frac{X_j X_{ji}X_i}{|X|}\bigg)+\frac{(n-1)q^2}{n} \sum_i u^{-q-1}|X|^{\frac{p}{p-1}}X_i u_i\non\\
    &=\sum_{i,j}\bigg(u^{1-q} X_{ij}X_{ji}
    -2 q u^{-q} X_{ij}X_j \tl{u}_i\bigg)+\frac{(n-1)q^2}{n}\sum_i u^{-q-1}|X|^{\frac{p}{p-1}}X_i u_i\non\\
    &+(1-q)\sum_{i,j} u^{-q} X_{ij}X_j (u_i-\tl{u}_i)+(q+1)P^k\non\\
    &=\sum_{i,j}u^{1-q} (X_{ij}-q L_{ij})(X_{ji}-q L_{ji})+E^k,\label{e90}
\end{align}
where 
\begin{align}
P^k=&\sum_{i,j}\bigg(u^{-q} X_{ij}X_j \tl{u}_i-u^{-q}|X|^{\frac{1}{p-1}}\frac{X_j X_{ji}X_i}{|X|}\bigg)\non\\
=&u^{-q}\sum_{i,j}X_{ij}X_j(\tl{u}_i-|X|^{\frac{1}{p-1}}\frac{\tl{u}_i}{|\n \tl u|})\label{e90.5}\\
    E^k=&\sum_{i,j}\bigg((1-q) u^{-q} X_{ij}X_j (\tl{u}_i-u_i)-2 q u^{1-q} X_{ij}D_{ij}
    \bigg)\non\\
    &+\frac{(n-1)q^2}{n}\sum_i u^{-q-1}|X|^{\frac{p}{p-1}}X_i (u_i-\tl{u}_i)+(q+1)P^k\non\\
    &+\frac{(n-1)q^2}{n}\sum_i u^{-q-1}|X|^{\frac{p}{p-1}}X_i \tl{u}_i-q^2\sum_{i,j}u^{1-q} L_{ij}L_{ji}.\label{e91}
    \end{align}
    Note that
    \begin{align}
   G^k:&=\frac{(n-1)q^2}{n}\sum_i u^{-q-1}|X|^{\frac{p}{p-1}}X_i \tl{u}_i-q^2\sum_{i,j}u^{1-q} L_{ij}L_{ji} \non\\
    &=\frac{(n-1)q^2}{n}\sum_i u^{-q-1}\big(|X|^{\frac{p}{p-1}}-(\frac{1}{k^2}+|\n \tl{u}|^2)^\frac{p}{2}\big)X_i \tl{u}_i\non\\
    &+q^2 u^{-q-1}\frac{1}{k^2}(\frac{1}{k^2}+|\n \tl{u}|^2)^{p-2}[(1-\frac{1}{n})|\n \tl{u}|^2-\frac{1}{n}\frac{1}{k^2}]+F^k,\label{e92}
\end{align}
where 
\begin{align}
    F^k=-q^2\sum_{i,j}u^{1-q} D_{ij}D_{ji}+2 q^2\sum_{i,j}u^{1-q}(\frac{X_i \tl{u}_j}{u}-\frac{(\frac{1}{k^2}+|\n \tl{u}|^2)^\frac{p}{2}}{n u}\d_{i j})D_{ij}.\label{e93}
\end{align}
Since $u>0$, $u_k\in C^{1,\a}_{loc}(\ov \O)$, uniformly in $k$, and $D^k\ra 0$ uniformly on compact subset of $\ov\O$, we have $F^k\ra 0$ uniformly on compact subset of $\ov\O$ and so does $G^k$.

Note that on compact subset of $\ov\O$, we have
$$
|P^k|\le C|\n X||X|\Big||X|^{\frac{1}{p-1}}-|\n \tl u|\Big|
$$
Because
$(u_k)_{k\in\nz}$
is bounded in
$C^1_{loc}(\ov \O)$ and $(\frac{1}{k^2}+|x|^2)^\frac{p-2}{2}x\ra |x|^{p-2}x$ uniformly on compact subset of $\rz^n$, as $k$ goes to infinity, we get
$$
|X^k|-|\n \tl u|^{p-1}\ra 0,\quad |X^k|-(\frac{1}{k^2}+|\n \tl{u}|^2)^\frac{p-1}{2}\ra 0\quad\T{in}\,\,C_{loc}^0(\ov\O)
$$
Since $(u_k)_{k\in\nz}$
is bounded in
$C^1_{loc}(\ov \O)$, $\{a^k(\n u_k)\}_{k\in\nz}$ is bounded in 
$W^{1,2}_{loc}(\ov\O)$, $x\mapsto x^{\frac{1}{p-1}}$ is locally Lipchitz continuous for $1<p<2$ and is uniformly continuous for $p\ge 2$, we have
$P^k\ra 0$ in 
$L^{2}_{loc}(\ov\O)$.

Therefore, using the fact that $(u_k)_{k\in\nz}$
converges in
$C^1_{loc}(\ov \O)$ to $u$, $D^k\ra 0$ uniformly on compact subset of $\ov\O$ and $\{a^k(\n u_k)\}_{k\in\nz}$ is bounded in 
$W^{1,2}_{loc}(\ov\O)$, we have $E^k\ra 0$ in 
$L^{2}_{loc}(\ov\O)$.

Second, using $(\ref{e87})$, we get the following Pohozaev-type differential equality
\begin{align}
    &\sum_i\p_n(X_i \tl{u}_i)-p\sum_i\p_i(X_i \tl{u}_n)\non\\
    =&\sum_i X_{i n}\tl{u}_i-(p-1)\sum_i X_i \tl{u}_{n i}\label{e94}
\end{align}
Since $X_{ij}=(\frac{1}{k^2}+|\n \tl{u}|^2)^{\frac{p-2}{2}}(\tl{u}_{ij}+\frac{(p-2)\tl{u}_i \tl{u}_l \tl{u}_{lj}}{\frac{1}{k^2}+|\n \tl{u}|^2})$, we have
\begin{align}
    \sum_i X_{i n}\tl{u}_i=&(\frac{1}{k^2}+|\n \tl{u}|^2)^{\frac{p-2}{2}}\bigg(1+\frac{(p-2)|\n\tl{u}|^2}{\frac{1}{k^2}+|\n \tl{u}|^2}\bigg)\sum_i \tl{u}_i \tl{u}_{n i}\non\\
    =&\bigg(1+\frac{(p-2)|\n\tl{u}|^2}{\frac{1}{k^2}+|\n \tl{u}|^2}\bigg)\sum_i X_i \tl{u}_{n i}\label{e95}
\end{align}
Using $(\ref{e94})$ and $(\ref{e95})$, we get
\begin{align}
    &\sum_i\p_n(X_i \tl{u}_i)-p\sum_i\p_i(X_i \tl{u}_n)\non\\
    =&-(p-2)\frac{1}{k^2}\frac{1}{\frac{1}{k^2}+|\n \tl{u}|^2}
    \sum_i X_i \tl{u}_{n i}\non\\
    =&-\bigg(1+\frac{(p-2)|\n\tl{u}|^2}{\frac{1}{k^2}+|\n \tl{u}|^2}\bigg)^{-1}\frac{(p-2)\frac{1}{k^2}}{\frac{1}{k^2}+|\n \tl{u}|^2}\sum_i X_{i n}\tl{u}_i:=I^k
    \label{e96}
\end{align}
Note that $$\min\{1,p-1\}\le (1+\frac{(p-2)|\n\tl{u}|^2}{\frac{1}{k^2}+|\n \tl{u}|^2})\le \max\{1,p-1\}.$$ 
Therefore, using the fact that $ab\le \frac{a^2+b^2}{2}$, we have
\begin{align*}
|I^k|\le& C_{n,p} \frac{\frac{1}{k^2}|\n \tl{u}|}{\frac{1}{k^2}+|\n \tl{u}|^2}|\n X|\\
\le &C_{n,p}\frac{1}{k}|\n X|
\end{align*}
Because $X^k\in W^{1,2}_{loc}(\ov \O)$, uniformly in k, we obtain $I^k\ra 0$ in $L^2_{loc}(\ov\O)$.\\
Then, it follows from $(\ref{e90})$ and $(\ref{e96})$ that
\begin{align}
    &\sum_i\p_i\bigg(\sum_j u^{1-q} X_{ij}X_j-\frac{(n-1)q}{n} u^{-q}|X|^{\frac{p}{p-1}}X_i\bigg)\non\\
    &+\frac{n-1}{n-p}\bigg[\sum_i\p_n(X_i \tl{u}_i)-p\sum_i\p_i(X_i \tl{u}_n)\bigg]\non\\
    &=\sum_{i,j}u^{1-q} (X_{ij}-q L_{ij})(X_{ji}-q L_{ji})+E^k+\frac{n-1}{n-p}I^k
    \label{e97}
\end{align}
Let $\eta\in C_c^\infty(B_R)$. We  multiply $(\ref{e97})$ by $\eta$ and integrate over $\O$. 

Then, it follows from the divergence theorem and $\nu=(0,\dots,0,-1)$ on $\p\O$ that
\begin{align}
    &\int_\O \bigg(
    \sum_{i,j}u^{1-q} (X_{ij}-q L_{ij})(X_{ji}-q L_{ji})+E^k+\frac{n-1}{n-p}I^k
    \bigg)\eta\, d x\non\\
    &=-\int_\O\sum_i\big(\sum_j u^{1-q} X_{ij}X_j-\frac{(n-1)q}{n} u^{-q}|X|^{\frac{p}{p-1}}X_i
    \big)\eta_i d x\non\\
    &-\frac{n-1}{n-p}\int_\O\big(\sum_i X_i \tl{u}_i\eta_n-p\sum_i X_i \tl{u}_n\eta_i
    \big) d x\non\\
    &-\int_{\p\O}\big(\sum_j u^{1-q} X_{nj}X_j-\frac{(n-1)q}{n} u^{-q}|X|^{\frac{p}{p-1}}X_n
    \big)\eta d\s\non\\
    &-\frac{n-1}{n-p}\int_{\p\O}\big(\sum_i X_i \tl{u}_i-p X_n \tl{u}_n
    \big)\eta d\s\label{e98}
\end{align}
Since $X^k_{n}=-u^q$ on $\p\O$ and $\sum_i X^k_{ii}=0$, we have
\begin{align}
    &\int_{\p\O}\big(\sum_j u^{1-q} X_{nj}X_j-\frac{(n-1)q}{n} u^{-q}|X|^{\frac{p}{p-1}}X_n
    \big)\eta d\s\non\\
=&\int_{\p\O}\big(-q\sum_{j=1}^{n-1} u_j X_j-u^{1-q}\sum_{j=1}^{n-1}X_{jj}X_n+\frac{(n-1)q}{n} |X|^{\frac{p}{p-1}}
    \big)\eta d\s\non\\
    =&\int_{\p\O}\big(-q\sum_{j=1}^{n-1} u_j X_j-\sum_{j=1}^{n-1}X_{j}u_j+\frac{(n-1)q}{n} |X|^{\frac{p}{p-1}}
    \big)\eta-\sum_{j=1}^{n-1}u X_{j}\eta_j d\s\non\\
    =&\int_{\p\O}\big(-\frac{(n-1)p}{n-p}\sum_{j=1}^{n-1} X_j\tl{u}_j +\frac{(n-1)q}{n} \sum_{j=1}^{n} X_j\tl{u}_j+J^k
    \big)\eta-\sum_{j=1}^{n-1}u X_{j}\eta_j d\s\non\\
    =&-\frac{n-1}{n-p}\int_{\p\O}\big(\sum_i X_i \tl{u}_i-p X_n \tl{u}_n
    \big)\eta d\s+\int_{\p\O}J^k\eta-\sum_{j=1}^{n-1}u X_{j}\eta_j d\s\label{e99}
\end{align}
where 
\begin{align}
    J^k=&\frac{(n-1)p}{n-p}\sum_{j=1}^{n-1} X_j(\tl{u}_j-u_j)+\frac{(n-1)q}{n}\big(|X|^\frac{p}{p-1}-(\frac{1}{k^2}+|\n \tl{u}|^2)^{\frac{p}{2}}\big)\non\\
    &+\frac{(n-1)q}{n}(\frac{1}{k^2}+|\n \tl{u}|^2)^{\frac{p}{2}}\frac{1/k^2}{\frac{1}{k^2}+|\n \tl{u}|^2}\non
\end{align}
Since $(u_k)_{k\in\nz}$ is bounded in $C^{1,\a}_{loc}(\ov \O)$, we can see that $(\frac{1}{k^2}+|\n \tl{u}|^2)^{\frac{p}{2}}\frac{\frac{1}{k^2}}{\frac{1}{k^2}+|\n \tl{u}|^2}\ra 0$ and $|X|^\frac{p}{p-1}\ra(\frac{1}{k^2}+|\n \tl{u}|^2)^{\frac{p}{2}}$ uniformly on compact subset of $\ov\O$.

Also, using the fact that $(u_k)_{k\in\nz}$
converges in
$C^1_{loc}(\ov \O)$ to $u$, we obtain $J^k\ra 0$ uniformly on compact subset of $\ov\O$.

It follows from $(\ref{e98})$ and $(\ref{e99})$ that
\begin{align}
    \int_\O &\bigg(
    \sum_{i,j}u^{1-q} (X_{ij}-q L_{ij})(X_{ji}-q L_{ji})
    \bigg)\eta\, d x\non\\
    =-&\int_\O\sum_{i,j}\big( u^{1-q} X_{ij}X_j-\frac{(n-1)q}{n} u^{-q}|X|^{\frac{p}{p-1}}X_i
    \big)\eta_i d x-\int_{\p\O}J^k\eta-u X_{j}\eta_j d\s\non\\
    &-\frac{n-1}{n-p}\int_\O\big(\sum_i X_i \tl{u}_i\eta_n-p\sum_i X_i \tl{u}_n\eta_i
    \big) d x-\int_{\O}\big(E^k+\frac{n-1}{n-p}I^k \big)\eta d x \label{e100}
\end{align}
Using $(\ref{e89})$ and $(\ref{e100})$, we get
\begin{align}
    \sum_{i,j}&\int_\O 
    u^{1-q} |X^k_{ij}-q L^k_{ij}|^2
    \eta\, d x\non\\
    \le -C_{p}^{-1}&\sum_{i,j}\int_\O u^{1-q} X^k_{ij}X^k_j\eta_i d x+C_{n,p}\bigg(\int_\O u^{-q}|X^k|^{\frac{2 p-1}{p-1}}|\n \eta| d x+\int_{\p\O}|J^k|\eta d\s\non\\
    &+\int_\O |X^k|^{\frac{p}{p-1}}|\n\eta|
     d x+\int_{\p\O} u|X^k||\n\eta|
    d\s+\int_{\O}\big(|E^k|+|I^k| \big)\eta d x \bigg)\label{e101}
\end{align}
Note that $(u_k)_{k\in\nz}$
converges in
$C^1_{loc}(\ov \O)$ to $u$ and $a^k(\n u_k)\ru a(\n u)$ in 
$W^{1,2}_{loc}(\ov\O)$, as $k\ra\infty$. Besides, we have shown $E^k,I^k\ra 0$ in $L^2_{loc}(\ov\O)$ and $J^k\ra 0$ uniformly on compact subset of $\ov\O$.
So, by passing to the limit as $k\ra\infty$ into $(\ref{e101})$ , we obtain
\begin{align}
&\limsup_{k\ra\infty}\sum_{i,j}\int_\O 
    u^{1-q} |X^k_{ij}-q L^k_{ij}|^2
    \eta\, d x\non\\
    &\le C_{p}^{-1}\bigg(-\sum_{i,j}\int_\O u^{1-q} X_{ij}X_j\eta_i d x\bigg)+C_{n,p}\bigg(\int_\O u^{-q}|\n u|^{2 p-1}|\n \eta| d x\non\\
    &+\int_\O |\n u|^{p}|\n\eta|
     d x +\int_{\p\O} u|\n u|^{p-1}|\n\eta|
     d \s\bigg)\non\\
    &\le C_{n,p}\bigg(\int_\O u^{1-q} |\n a(\n u)|\,|\n u|^{p-1}|\n\eta| d x+\int_\O u^{-q}|\n u|^{2 p-1}|\n \eta| d x\non\\
    &+\int_\O |\n u|^{p}|\n\eta|
     d x+\int_{\p\O} u|\n u|^{p-1}|\n\eta|
     d \s \bigg).
    \label{e102}
\end{align}
Let $R>1$ and $\eta$ be a non-negative cut-off function such that $\eta=1$ in $B_R$, $\eta=0$ outside $B_{2R}$ and $|\n \eta|\le \frac{C_n}{R}$. It follows from Proposition \ref{p1}, Proposition \ref{p2} and $(\ref{e102})$ that
\begin{align}
\limsup_{k\ra\infty}\sum_{i,j}\int_\O 
    u^{1-q} |X^k_{ij}-q L^k_{ij}|^2
    \eta\, d x\le C R^{\frac{p-n}{p-1}}.
    \label{e103}
\end{align}
Since $X^k_{ij}-q L^k_{ij}\ru X_{ij}-q L_{ij}$ in 
$L^{2}_{loc}(\ov\O)$, $u^{\frac{1-q}{2}}\eta^{\frac{1}{2}}\in L^\infty(\ov\O)$, we have \begin{align}  
u^{\frac{1-q}{2}}\eta^{\frac{1}{2}}(X^k_{ij}-q L^k_{ij})\ru u^{\frac{1-q}{2}}\eta^{\frac{1}{2}}(X_{ij}-q L_{ij})\quad\T{in}\quad L^{2}(\ov\O).\label{e104}
\end{align}
Therefore, it follows from $(\ref{e103})$ and $(\ref{e104})$ that 
\begin{align}
\sum_{i,j}\int_\O 
    u^{1-q} |X_{ij}-q L_{ij}|^2
    \eta\, d x\le C R^{\frac{p-n}{p-1}}.
    \label{e105}
\end{align}
Letting $R\ra\infty$, we get $X_{ij}-q L_{ij}=0$ in $L^2_{loc}(\ov\O)$.
Since $u\in C^{1,\a}_{loc}(\ov\O)$, we have $L_{ij}\in C^{\a}_{loc}(\ov\O)$ and $X_i\in C^{\a}_{loc}(\ov\O)$,
which implies $X_i\in C^{1,\a}_{loc}(\ov\O)$. This ends the proof of Proposition \ref{p3}.
 
\end{proof}

\subsection{Proof of Theorem \ref{t1}}
We consider the auxiliary function
\begin{equation}
    v=\frac{n-p}{p}u^{-\frac{p}{n-p}}.\non
\end{equation}
Since $a(\n u)\in C^{1,\a}_{loc}(\ov\O)$ and $a(\n v)=-u^{-\frac{n(p-1)}{n-p}}a(\n u)$, we have $a(\n v)\in C^{1,\a}_{loc}(\ov\O)$. Then, it follows from Proposition \ref{p3} that
\begin{equation}
\p_j( a_i(\n v(x)))=\l(x)\d_{ij},\label{e106}
\end{equation}
where $\l(x)=\frac{1}{n}u^{-\frac{(n-1)p}{n-p}}|\n u|^p$, $i,j\in\{1,\dots,n\}$.\\
Since elliptic regularity theory
yields that $v\in C^{2,\a}_{loc}(\O\cap \{\n v\neq 0\})$, which implies that $\l(x)\in C^{1,\a}_{loc}(\O\cap \{\n v\neq 0\})$. Then, using $(\ref{e106})$, we get $a(\n v)\in C^{2,\a}_{loc}(\O\cap \{\n v\neq 0\})$.
Therefore, it follows from $(\ref{e106})$ that for $j\neq i$, we have
\begin{equation}
\p_j \l(x)=\p_j\p_i( a_i(\n v(x)))=\p_i\p_j( a_i(\n v(x)))=0\non
\end{equation}
for any $x\in\O\cap \{\n v\neq 0\}$, which implies that $\l(x)$ is constant on each connected component of
$\O\cap \{\n v\neq 0\}$. Note that using Proposition \ref{p2}, we obtain that there exists $R_0>1$ such that $\ov\O\bs B_{R_0}\subset\{\n v\neq 0\}$. So we consider the connected component $U$ of
$\ov\O\cap \{\n v\neq 0\}$ such that $\ov\O\bs B_{R_0}\subset U$. Since $\l\in C^{\a}_{loc}(\ov\O)$, we get $\l(x)$ is constant on $U$. Then, we deduce that
$$
a(\n v)=\l(x-x_0)\quad\T{in}\,\, \ov U
$$
for some $x_0\in \rz^n$. Using $(\ref{e81})$, we obtain
\begin{align}\label{e107}
    n \l=\div(a(\n v))=\frac{n(p-1)}{p}\frac{|\n v|^p}{v}\quad\T{in}\,\, \ov U.
\end{align}

Note that $\{a(\n v)\neq 0\}=\{\n v\neq 0\}$ and $a_n(\n v)=1$ on $\p\O$ implies that $x_0\notin U\cup\p\O$.
Since for any $y\in\p U$, we have either $y\in\p\O$ or $a(\n v(y))=0$, we deduce that $\p U\subset\p\O\cup \{x_0\}$. 

Now, we divide $x_0$ into two cases.\\
\textbf{Case 1:} $x_0\in \rz^n_+$.
Then $x_0^n>0$, $U=\ov\O\bs\{x_0\}$ and $a(\n v)=\l(x-x_0)$ in $\ov\O$. Using $a_n(\n v)=1$ on $\p\O$ and $(\ref{e107})$, we get
$\l=-\frac{1}{x_0^n}<0$ and $v<0$, which is in contradiction with $u>0$.

\noindent\textbf{Case 2:} $x_0\in \rz^n_-$. Then $x_0^n<0$, $U=\ov\O$ and $a(\n v)=\l(x-x_0)$ in $\ov\O$. Since $a_n(\n v)=1$ on $\p\O$, we get
$\l=-\frac{1}{x_0^n}>0$.
The fact that $|\n v|^{p-2}\n v=\l (x-x_0)$ in $\ov\O$ implies 
$\n v=(\l |x-x_0|)^\frac{1}{p-1}\frac{x-x_0}{|x-x_0|}$. Therefore, it follows from $(\ref{e107})$ that 
$$
v(x)=\frac{p-1}{p}\l^\frac{1}{p-1}|x-x_0|^\frac{p}{p-1},
$$
where $\l=-\frac{1}{x_0^n}>0$.
Using $v=\frac{n-p}{p}u^{-\frac{p}{n-p}}$, we get Theorem \ref{t1}.

\section{Application}

Adapting the proof of Theorem \ref{t1}, we can get the following Theorem.

\begin{thm}\label{t2}
Let $n\ge 2,\,1<p<n,\,d>0$ and let $u$ be a solution to 
\begin{equation}\label{e109}
    \left\{
    \begin{array}{lr}
	\Delta_p u+d u^{p^*-1}=0 & \T{in} \quad \mathbb{R}^n_+\\
	|D u|^{p-2}\frac{\partial u}{\partial t}=\mp u^q &\T{on} \quad \p\mathbb{R}^n_+ \\
    u>0 & \T{in} \quad \mathbb{R}^n_+\\
     u\in D^{1,p}(\mathbb{R}^n_+),
	\end{array}
 \right.     
\end{equation}
Then 
$$
u(y,t)=\bigg(\frac{n-p}{p}\bigg)^{\frac{n-p}{p}}\bigg(\frac{\frac{n(p-1)}{p}\l^{\frac{p}{p-1}}|x-x_0|^{\frac{p}{p-1}}+d\frac{n-p}{p}}{n\l}\bigg)^{-\frac{n-p}{p}}
$$
for some $x_0\in \rz^{n}_{\mp}$ and $\l=\mp\frac{1}{x_0^n}>0$.
\end{thm}
 First, we prove that solutions to $(\ref{e109})$ are
bounded. The result holds for more general Neumann problems, in particular for problems with
a differential operator modelled on the p-Laplace operator.
\begin{lem}\label{l5}
    Suppose $u\in D^{1,p}(\O)$ is a solution to 
\begin{equation}\label{e110}
    \left\{
    \begin{array}{lr}
	\div (a(\n u))=f(x,u) & ,\T{in} \,\,\O\\
	a(\n u)\cdot \nu=\Phi(x,u) & ,\T{on} \,\, \p \O\\
    u>0 & ,\T{in} \,\,\O
	\end{array}
 \right.
\end{equation}
where $a:\rz^n\ra\rz^n$ is a continuous vector field such that the following holds: there exist $\a>0,\,\l,\g,\L,\mu\ge 0,$ and $0\le s\le 1/2$ such that 
\begin{equation}\label{e111}
    \left\{
    \begin{array}{lr}
	|a(x)|\le \a (|x|^2+s^2)^{\frac{p-1}{2}}&,\,x\in \rz^n \\
	x\cdot a(x)\ge \frac{1}{\a}\int_0^1(t^2|x|^2+s^2)^{\frac{p-2}{2}}|x|^2 dt&,\,x\in \rz^n\\
    |\Phi(x,z)|\le \l |z|^q+\g,\,|f(x,z)|\le \mu |z|^{p^*-1}+\L&,\,x\in \rz^n,\,z\in \rz.
	\end{array}
 \right.
\end{equation}
Then there exist $\d>0$ with the following property: let $0<\rho\le c_n$, $x_0\in\O$ be such that 
\begin{align*}
    &\T Vol (B_{c_n})\le 1,\,\T Area (B_{c_n}\cap \p\O)\le 1,\\
    &\|u\|_{L^{p_*}(B_{\rho}(x_0)\cap \p\O)}\le \d,\,\|u\|_{L^{p^*}(B_{\rho}(x_0)\cap \O)}\le \d
\end{align*}
   then for any $t>0$
   $$\|u\|_{L^\infty(B_{\rho/4}(x_0)\cap \O)}\le C\big(\|u\|_{L^t(B_{\rho/2}(x_0)\cap \O)}+\|u\|_{L^t(B_{\rho/2}(x_0)\cap \p\O)}+1\big)$$
   where $C$ depend only on $n,p,\a,\rho,\l,t,\g,\mu,\L$; and $\d$ depend only on $n,p,\a,\l,\mu$.
\end{lem}
\begin{proof}
The proof of Lemma \ref{l5} is similar to Lemma \ref{l1}.
We point out hereafter the changes required for the solution to $(\ref{e110})$.\\
We argue as Lemma \ref{l1}. Note that $(\ref{e6})$ becomes
\begin{align}
    &(\b-1)\int_\O a(\n \ov u)\cdot\n \ov u \,\eta^p\ov u^{\b-1} d x+\int_\O a(\n u)\cdot\n u \,\eta^p\ov u^{\b-1} d x\non\\
    =&-p\int_\O a(\n u)\cdot\n\eta\,\eta^{p-1}(\ov u^{\b-1}u_k-k^\b) dx+\int_{\p\O} \Phi(x,u)\eta^p(\ov u^{\b-1}u_k-k^\b) d\s\non\\
    &-\int_{\O} f(x,u)\eta^p(\ov u^{\b-1}u_k-k^\b) d x.\label{e112}
\end{align}
And $(\ref{e8})$ becomes
\begin{align}
    &(\b-1)\int_\O |\n \ov u |^p\,\eta^p\ov u^{\b-1} d x+
    \int_\O |\n u |^p\,\eta^p\ov u^{\b-1} d x\non\\
    \le& C_1\bigg(\int_{\p\O} (\l u_k^q+\g)\eta^p\ov u^{\b-1}u_k d\s
    +\int_\O |a(\n u)||\n\eta|\,\eta^{p-1}\ov u^{\b-1}u_k dx\non\\&+\int_{\O} (\mu u_k^{p^*-1}+\L)\eta^p\ov u^{\b-1}u_k d x+\b\int_{\O}\eta^p\ov u^{\b-1} d x\bigg)
    \label{e113}
\end{align}
where $C_1$ always depends only on $\a,p$.\\
Also, $(\ref{e9})$ becomes
\begin{align}
    \int_\O& |\n(\eta \ov u^{\frac{\b-1}{p}}u_k)|^p d x\le
    C_1\bigg(
    \b^{p-1}\int_\O \ov u^{\b-1}u_k^p|\n \eta|^p d x+(\frac{\b}{k})^p\int_{\O}\eta^p\ov u^{\b-1}u_k^p d x\non\\&+\g(\frac{\b}{k})^{p-1}\int_{\p\O}\eta^p\ov u^{\b-1}u_k^p d\s+\l\b^{p}\int_{\p\O}\eta^p\ov u^{\b-1}u_k^{q+1} d\s\non\\
    &+\L(\frac{\b}{k})^{p-1}\int_{\O}\eta^p\ov u^{\b-1}u_k^p d x
    +\mu \b^{p}\int_{\O}\eta^p\ov u^{\b-1}u_k^{p^*} d x
    \bigg)\label{e114}
\end{align}
Therefore, using H\"{o}lder's inequality, Sobolev trace inequality and Observation \ref{o2}, $(\ref{e10})$ becomes
\begin{align}  
\int_{\p\O}\eta^p\ov u^{\b-1}u_k^{q+1} d\s
\le& C_{n,p} \,(\d+k)^{\frac{p(p-1)}{n-p}}\int_\O |\n(\eta \ov u^{\frac{\b-1}{p}}u_k)|^p d x\label{e115}\\
\int_{\O}\eta^p\ov u^{\b-1}u_k^{p^*} d x\le
&\| u_k\|_{L^{p^*}(B_{\rho}(x_0)\cap \O)}^{\frac{p^2}{n-p}}\,\|\eta \ov u^{\frac{\b-1}{p}}u_k\|_{L^{p^*}(\O)}^p\non\\
\le& C_{n,p} \,(\d+k)^{\frac{p^2}{n-p}}\int_\O |\n(\eta \ov u^{\frac{\b-1}{p}}u_k)|^p d x\label{e116}
\end{align}
First, we fix $\b=p_*-p+1,\,k=\d$ and choose $\d>0$ small enough which depends only on $\l,\mu,n,p,\a$. Then, from Sobolev trace inequality, Observation \ref{o2}, $(\ref{e114})$, $(\ref{e115})$ and $(\ref{e116})$, we have
    \begin{align}
    \|\eta^p\ov u^{p_*-p}u_k^p\|_{L^a(\O)}+&\|\eta^p\ov u^{p_*-p}u_k^p\|_{L^b(\p\O)}\le
    C_2\bigg(
    \int_\O \ov u^{p_*-p}u_k^p|\n \eta|^p d x\non\\&+(\L+1)\b^p\int_{\O}\eta^p\ov u^{p_*-p}u_k^p d x+\g\b^{p}\int_{\p\O}\eta^p\ov u^{p_*-p}u_k^p d\s\bigg)\non
\end{align}
Since $\rho\le c_n,\,\T Vol (B_{c_n})\le 1,\,\T Area (B_{c_n}\cap \p\O)\le 1$, here we choose $\eta$ such that $\eta=1$ in $B_{\rho/2}(x_0)$ and $|\n \eta|\le C_n/\rho$, then by H\"{o}lder's inequality and $u_k\ge \ov u$, we have
\begin{align}
    &\|\ov u^{p_*-p}u_k^p\|_{L^a(\O\cap B_{\rho/2}(x_0))}+\|\ov u^{p_*-p}u_k^p\|_{L^b(\p\O\cap B_{\rho/2}(x_0))}\non\\
    &\le C_3\bigg(\| u_k^{p_*}\|_{L^{\frac{a}{b}}(\O\cap B_{\rho}(x_0))}
+\|u_k^{p_*}\|_{L^{1}(\p\O\cap B_{\rho}(x_0))}\bigg)\non\\
&\le C_3\bigg((\|u\|_{L^{p^*}(\O\cap B_{\rho}(x_0))}+k)^{p_*}
+\g(\| u\|_{L^{p_*}(\p\O\cap B_{\rho}(x_0))}+k)^{p_*}\bigg)\non
\end{align}
So, letting $l\ra\infty$, we have
\begin{align}
\RA &\|u_k\|_{L^{p_* a}(\O\cap B_{\rho/2}(x_0))}+\|u_k\|_{L^{p_* b}(\p\O\cap B_{\rho/2}(x_0))}\non\\
\le &C_3\bigg(\|u\|_{L^{p^*}(\O\cap B_{\rho}(x_0))}+\| u\|_{L^{p_*}(\p\O\cap B_{\rho}(x_0))}+\d\bigg)\label{e117}\\
\le &C_3\non
\end{align}
where $C_3$ always depends only on $\a,n,p,\l,\mu,\rho,\L,\g$.\\
Second, let $\b$ vary, choose $\eta$ such that $\operatorname{supp}(\eta)\subset B_{\rho/2}(x_0)$ and fix $k=\d$, we come back to $(\ref{e114})$. Note that $(\ref{e12})$ becomes
\begin{align}  
\int_{\p\O}\eta^p\ov u^{\b-1}u_k^{q+1} d\s\le
&\|u_k\|_{L^{p_* b}(B_{\rho/2}(x_0)\cap \p\O)}^{\frac{p(p-1)}{n-p}}\,\|\eta^p \ov u^{\b-1}u_k^p\|_{L^{t_1}(\p\O)}\non\\
\le
&C_3^{\frac{p(p-1)}{n-p}}\|\eta^p \ov u^{\b-1}u_k^p\|_{L^{t_1}(\p\O)}\label{e118}\\
\int_{\O}\eta^p\ov u^{\b-1}u_k^{p^*} d x\le
&\| u_k\|_{L^{p_* a}(B_{\rho/2}(x_0)\cap \O)}^{\frac{p^2}{n-p}}\,\|\eta^p \ov u^{\b-1}u_k^p\|_{L^{t_2}(\O)}\non\\
\le
&C_3^{\frac{p^2}{n-p}}\|\eta^p \ov u^{\b-1}u_k^p\|_{L^{t_2}(\O)}\label{e119}
\end{align}
where $\frac{1}{t_1}+\frac{b-1}{b^2}=1,\,\frac{1}{t_2}+\frac{a-1}{a b}=1\RA t_1\in(1,b),\,t_2\in(1,a)$.\\
Then using Observation \ref{o3}, Sobolev trace inequality and Observation \ref{o2}, $(\ref{e13})$ becomes
\begin{align}  
\int_{\p\O}\eta^p\ov u^{\b-1}u_k^{q+1} d\s
\le
& C_{n,p}\e \b^{-p} \int_\O |\n(\eta \ov u^{\frac{\b-1}{p}}u_k)|^p d x+ C_4\b^{\s_1}\|\eta^p \ov u^{\b-1}u_k^p\|_{L^{1}(\p\O)}
\label{e120}\\
\int_{\O}\eta^p\ov u^{\b-1}u_k^{p^*} d x
\le& \e \b^{-p} \|\eta^p \ov u^{\b-1}u_k^p\|_{L^{a}(\O)}+ C_4\b^{\s_2}\|\eta^p \ov u^{\b-1}u_k^p\|_{L^{1}(\O)}\non\\
\le& C_{n,p}\e \b^{-p} \int_\O |\n(\eta \ov u^{\frac{\b-1}{p}}u_k)|^p d x+ C_4\b^{\s_2}\|\eta^p \ov u^{\b-1}u_k^p\|_{L^{1}(\O)}
\label{e121}
\end{align}
where $\s_1=\frac{p(b-t_1)}{b(t_1-1)},\s_2=\frac{p(a-t_2)}{a(t_2-1)}$ and $C_4\sim C_3,\e$.\\
Then we choose $\e>0$ small enough depending only on $n,p,\a,\l,\mu$. From $(\ref{e114})$, $(\ref{e120})$ and $(\ref{e121})$, we get
\begin{align}
    \int_\O |\n(\eta \ov u^{\frac{\b-1}{p}}u_k)|^p d x\le&
    C_3\b^\s\bigg(
    \int_\O \ov u^{\b-1}u_k^p|\n \eta|^p d x+\int_{\O}\eta^p\ov u^{\b-1}u_k^p d x\non\\&+\int_{\p\O}\eta^p\ov u^{\b-1}u_k^p d\s\bigg),\label{e122}
\end{align}
where $\s=\max\{\s_1,\s_2,p\}$.\\
Since $\rho\le c_n,\,\T Vol (B_{c_n})\le 1,\,\T Area (B_{c_n}\cap \p\O)\le 1$, here we choose $\eta$ such that $\eta=1$ in $B_{r}(x_0)$, $\eta=0$ outside $B_{R}(x_0)$ and $|\n \eta|\le C_n/(R-r)$, where $0<r<R<\rho/2$. Then using H\"{o}lder's inequality, $(\ref{e15})$ becomes
\begin{align}
    \|\ov u^{\b-1}u_k^p\|_{L^b(\O\cap B_r(x_0))}+\|\ov u^{\b-1}u_k^p&\|_{L^b(\p\O\cap B_r(x_0))}\le
    \non\\C_3\b^\s\bigg(\frac{1}{(R-r)^p}\|\ov u^{\b-1}u_k^p\|_{L^1(\O\cap B_R(x_0))}
+&\|\ov u^{\b-1}u_k^p\|_{L^1(\p\O\cap B_R(x_0))}\bigg)\label{e123}
\end{align}
Finally, arguing as Lemma \ref{l1}, we finish the proof.
    
\end{proof}

\begin{rem}\label{r2.5}
Similar to Remark \ref{r2.2}, if $s=\g=\L=0$, we can choose $k=0$, and Lemma \ref{l5} becomes: for any $t>0$
   $$\|u\|_{L^\infty(B_{\rho/4}(x_0)\cap \O)}\le C\|u\|_{L^t(B_{\rho/2}(x_0)\cap \O)}.$$
\end{rem}

Second, we consider the asymptotic estimates on $u$. For $R>1,\,x\in \O$, we define 
\begin{equation}
    u_R(x):=R^\frac{n-p}{p-1}u(R x).\non
\end{equation}
Then, from $(\ref{e109})$, we have
\begin{equation}\label{e124}
    \left\{
    \begin{array}{lr}
	\Delta_p u_R+d R^{-\frac{p}{p-1}}u_R^{p^*-1}=0 & , \T{in} \,\,\O\\
	|D u_R|^{p-2}\frac{\partial u_R}{\partial t}=\mp R^{-1}u_R^q & , \T{on} \,\,\p\O
	\end{array}
 \right.     
\end{equation}
Denote by $U:=\{1<|x|<8\}$, $W_1:=\{2<|x|<7\}$ and $W:=\{3<|x|<6\}$.
\begin{lem}\label{l6}
Suppose $u\in W^{1,p}\cap L^\infty(U\cap \O )$ is a positive solution to
    \begin{equation}\label{e125}
    \left\{
    \begin{array}{lr}
	\Delta_p u=-g(x)u^{p-1} & ,\T{in} \,\,\O\cap U\\
	|D u|^{p-2}\frac{\partial u}{\partial t}=-f(x) u^{p-1} & ,\T{on} \,\,\p\O\cap U
	\end{array}
 \right.     
\end{equation}
where $|g(x)|\le \mu |u(x)|^{p^*-p}$ in $U\cap \O$, $|f(x)|\le \l|u(x)|^{p_*-p}$ in $U\cap\p\O$ and $\mu,\l\ge 0$.\\
Then there exist $\d>0$ with the following property: let $c_n\in(0,1),\l,\mu$ be such that 
\begin{align*}
    \T Vol (B_{c_n})&\le 1,\,\,\T Area (B_{c_n}\cap \p\O)\le 1,
    \,\,\l \|u\|_{L^{p_*}(U\cap\p\O)}^\frac{p(p-1)}{n-p}\le \d,\\
    &\l \|u\|_{L^{p^*}(U\cap\O)}^\frac{p(p-1)}{n-p}\le \d,\,\,\mu \|u\|_{L^{p^*}(U\cap\O)}^\frac{p^2}{n-p}\le \d,
\end{align*}
   then for any $t>0$, $0<\rho\le c_n$, $x_0\in\O\cap W$
   $$\|u\|_{L^\infty(B_{\rho/2}(x_0)\cap \O)}\le C\|u\|_{L^t(B_{\rho}(x_0)\cap \O)}$$
   where $C$ depend only on $n,p,\rho,t$; and $\d$ depend only on $n,p$.
\end{lem}
\begin{proof}
    The proof of Lemma \ref{l6} is similar to the one of Lemma \ref{l2} and we only point out hereafter the changes required for the solution to $(\ref{e125})$.\\
We argue as Lemma \ref{l2}. Let $\eta\in C_c^\infty(U),\,\b\ge 1$ and use $\vp=\eta^p u^\b$ as a test function in $(\ref{e125})$, then we get
\begin{align}
     \b\int_\O |\n u|^p&\,\eta^p u^{\b-1} d x=-p\int_\O |\n u|^{p-2}\n u \cdot\n\eta\,\eta^{p-1} u^\b dx\non\\&+\int_{\p\O} f(x)\eta^p  u^{\b+p-1} d\s+\int_{\O} g(x)\eta^p  u^{\b+p-1} d x\non
     \end{align}
So $(\ref{e23})$ becomes
\begin{align}
     &\int_\O |\n(\eta  u^{\frac{\b+p-1}{p}})|^p d x\le C_1\bigg(\int_\O |\n \eta|^p u^{\b+p-1} dx\non\\&+\b^p\int_{\p\O} |f(x)|\eta^p  u^{\b+p-1} d\s+\b^p\int_{\O} |g(x)|\eta^p  u^{\b+p-1} d x\bigg),
     \label{e126}
\end{align}
where $C_1$ always depend only on $p$.\\
Note that 
\begin{align*}
&\|f\|_{L^{\frac{n-1}{p-1}}(U\cap \p\O)}\le \l \|u\|_{L^{p_*}(U\cap\p\O)}^\frac{p(p-1)}{n-p}\le \d,\\
&\|g\|_{L^{\frac{n}{p}}(U\cap \O)}\le \mu \|u\|_{L^{p^*}(U\cap\O)}^\frac{p^2}{n-p}\le \d.
\end{align*}
So, it follows from H\"{o}lder's inequality, Sobolev trace inequality and Sobolev inequality that $(\ref{e24})$ becomes
\begin{align}  
\int_{\p\O}|f(x)|\eta^p  u^{\b+p-1} d\s
\le& C_{n,p} \,\d\int_\O |\n(\eta u^{\frac{\b+p-1}{p}})|^p d x,\label{e127}\\
\int_{\O}|g(x)|\eta^p  u^{\b+p-1} d x\le
&\|g\|_{L^{\frac{n}{p}}(U\cap\O)}\,\|\eta u^{\frac{\b+p-1}{p}}\|_{L^{p^*}(\O)}^p\non\\
\le& C_{n,p} \,\d\int_\O |\n(\eta u^{\frac{\b+p-1}{p}})|^p d x.\label{e128}
\end{align}
First, we fix $\b=p_*-p+1$ and choose $\d>0$ small enough which depends only on $n,p$. Then, from $(\ref{e126})$, $(\ref{e127})$ and $(\ref{e128})$, we have
    \begin{align}
    \|\eta^p u^{p_*}\|_{L^a(\O)}+\|\eta^p u^{p_*}\|_{L^b(\p\O)}\le
    C_2
    \int_\O u^{p_*}|\n \eta|^p d x,\non
\end{align}
where $C_2$ always depend only on $n,p$.\\
Here we choose $\eta$ such that $\eta=1$ in $W_1$, $\eta=0$ outside $U$ and $|\n \eta|\le C_n$, then by H\"{o}lder's inequality we have
\begin{align}
        &\| u\|_{L^{p_* b}(\p\O\cap W_1)}\le C_2\| u\|_{L^{p_*}(\O\cap U)}\le C_2\| u\|_{L^{p^*}(\O\cap U)},\non\\
        &\| u\|_{L^{p^* b}(\O\cap W_1)}\le C_2\| u\|_{L^{p_*}(\O\cap U)}\le C_2\| u\|_{L^{p^*}(\O\cap U)}.\non
    \end{align}
    Therefore, we get
\begin{align}
        \|f\|_{L^{\frac{n-1}{p-1}b}(W_1\cap \p\O)}\le
        \l \|u\|_{L^{p_* b}(W_1\cap\p\O)}^\frac{p(p-1)}{n-p}\le C_2 \l \|u\|_{L^{p^*}(U\cap\O)}^\frac{p(p-1)}{n-p}\le C_2,\label{e129}\\
        \|g\|_{L^{\frac{n}{p}b}(W_1\cap \O)}\le
        \mu \|u\|_{L^{p^* b}(W_1\cap\O)}^\frac{p^2}{n-p}\le C_2 \mu \|u\|_{L^{p^*}(U\cap\O)}^\frac{p^2}{n-p}\le C_2.\label{e130}
    \end{align}
Second, let $\b$ vary and $\operatorname{supp}\eta\subset W_1$, we come back to $(\ref{e126})$. Like $(\ref{e26})$, we have
\begin{align}  
\int_{\p\O}|f(x)|\eta^p  u^{\b+p-1} d\s\le
&\|f\|_{L^{\frac{n-1}{p-1}b}(W_1\cap \p\O)}\,\|\eta^p u^{\b+p-1}\|_{L^{t_1}(\p\O)}\non\\
\le& C_2\|\eta^p u^{\b+p-1}\|_{L^{t_1}(\p\O)},\label{e131}\\
\int_{\O}|g(x)|\eta^p  u^{\b+p-1} d x\le
&\|g\|_{L^{\frac{n}{p}b}(W_1\cap \O)}\,\|\eta^p u^{\b+p-1}\|_{L^{t_2}(\O)}\non\\
\le& C_2\|\eta^p u^{\b+p-1}\|_{L^{t_2}(\O)},\label{e132}
\end{align}
where $\frac{1}{t_1}+\frac{b-1}{b^2}=1,\,\frac{1}{t_2}+\frac{a-1}{a b}=1\RA t_1\in(1,b),\,t_2\in(1,a)$.\\
Then, it follows from Observation \ref{o3}, Sobolev trace inequality and Sobolev inequality that $(\ref{e27})$ becomes
\begin{align}  
\int_{\p\O}|f(x)|\eta^p  u^{\b+p-1} d\s
\le& \e \b^{-p} \|\eta^p u^{\b+p-1}\|_{L^{b}(\p\O)}+ C_3\b^{\s_1}\|\eta^p u^{\b+p-1}\|_{L^{1}(\p\O)}\non\\
\le C_{n,p}\e &\b^{-p} \int_\O |\n(\eta u^{\frac{\b+p-1}{p}})|^p d x+ C_3\b^{\s_1}\|\eta^p u^{\b+p-1}\|_{L^{1}(\p\O)},
\label{e133}\\
\int_{\O}|g(x)|\eta^p  u^{\b+p-1} d\s
\le& \e \b^{-p} \|\eta^p u^{\b+p-1}\|_{L^{a}(\O)}+ C_3\b^{\s_2}\|\eta^p u^{\b+p-1}\|_{L^{1}(\O)}\non\\
\le C_{n,p}\e &\b^{-p} \int_\O |\n(\eta u^{\frac{\b+p-1}{p}})|^p d x+ C_3\b^{\s_2}\|\eta^p u^{\b+p-1}\|_{L^{1}(\O)},
\label{e134}
\end{align}
where $\s_1=\frac{p(b-t_1)}{b(t_1-1)},\s_2=\frac{p(a-t_2)}{a(t_2-1)}$ and $C_3\sim n,p,\e$.\\
Then we choose $\e>0$ small enough depending only on $n,p$. It follows from $(\ref{e133}),\,(\ref{e134})$, $(\ref{e126})$ and $\b\ge 1$ that
\begin{align}
\int_\O |\n(\eta  u^{\frac{\b+p-1}{p}})|^p &d x\le C_2\b^\s\bigg(\int_\O |\n \eta|^p u^{\b+p-1} d x\non\\+&\int_{\p\O} \eta^p  u^{\b+p-1} d\s+\int_{\O} \eta^p  u^{\b+p-1} d x\bigg),
     \label{e135}
\end{align}
where $\s=\max\{\s_1,\s_2\}$.\\
Since $\rho\le c_n,\,\T Vol (B_{c_n})\le 1,\,\T Area (B_{c_n}\cap \p\O)\le 1$, here we choose $\eta$ such that $\eta=1$ in $B_{r}(x_0)$, $\eta=0$ outside $B_{R}(x_0)$ and $|\n \eta|\le C_n/(R-r)$, where $0<r<R<\rho$. Then by H\"{o}lder's inequality, $(\ref{e30})$ becomes
\begin{align}
    \|u^{\b+p-1}\|_{L^b(\O\cap B_r(x_0))}+\|u^{\b+p-1}&\|_{L^b(\p\O\cap B_r(x_0))}\le
    \non\\C_3\b^\s\bigg(\frac{1}{(R-r)^p}\| u^{\b+p-1}\|_{L^1(\O\cap B_R(x_0))}
+&\| u^{\b+p-1}\|_{L^1(\p\O\cap B_R(x_0))}\bigg)\label{e136}
\end{align}
where $C_3$ depends only on $n,p,\rho$.\\
Finally, arguing as Lemma \ref{l2}, we finish the proof.

\end{proof}

    \begin{rem}\label{r2.6}
    In our setting, $\l=R^{-1}$, $\mu=d R^{-\frac{p}{p-1}}$, $g(x)=d R^{-\frac{p}{p-1}}u^{p^*-p}$ and $f(x)=\pm R^{-1}u^{p_*-p}$. Arguing as Remark \ref{r2.4}, we obtain that there exists $R_2>1$ such that
$$
|x|^\frac{n-p}{p}u(x)+|x|^\frac{n}{p}|\n u(x)|\le C\,\,\T{in}\,\,\O\cap \{|x|>R_2\}
$$

\end{rem}

\begin{lem}\label{l7}
    Any solution of $(\ref{e109})$ belongs to $L^{(p-1)a,\infty}(\O)$ and $L^{(p-1)b,\infty}(\p\O)$.
\end{lem}

\begin{proof}
    Adapting the proof of Lemma \ref{l3}, we obtain this lemma.\\
For every $h>0$, we define $u_h=\min\{u,h\}$, $W_h=\{u>h\}$.
Similar to Lemma \ref{l3}, we use $\vp=u_h$ as a test function in $(\ref{e109})$, then we get
\begin{align}
    \int_{\O} |\n u_h|^p d x=&\pm\int_{\p\O} u^q u_h d\s+d\int_{\O} u^{p^*-1} u_h d x\non\\
    =&\pm\int_{\p\O\cap W_h^c} u^{p_*} d\s\pm h\, \int_{\p\O\cap W_h} u^q d\s\non\\
    +&d\int_{\O\cap W_h^c} u^{p^*} d x+d h\, \int_{\O\cap W_h} u^{p^*-1} d x\label{e137}
\end{align}
On the one hand, straightforward computations give
\begin{align}
    \int_{\p\O\cap W_h^c} u^{p_*} d\s=\int_{\p\O}u_h^{p_*} d \s-h^{p_*}\,\H^{n-1}(W_h\cap\p\O)\label{e138}\\
    \int_{\O\cap W_h^c} u^{p^*} d x=\int_{\O}u_h^{p^*} d x-h^{p^*}\,\H^{n}(W_h\cap\O)\label{e139}
\end{align}
On the other hand, using $(\ref{tool1})$ we get 
\begin{align}
    \int_{\p\O\cap W_h} u^q d\s
    =&h^q \H^{n-1}(W_h\cap\p\O)+q\int_h^\infty s^{q-1}\H^{n-1}(\{u>s\}\cap\p\O)\, d s\label{e140}\\
    \int_{\O\cap W_h} u^{p^*-1} d x
    =&h^{p^*-1} \H^{n}(W_h\cap\O)+(p^*-1)\int_h^\infty s^{p^*-2}\H^{n}(\{u>s\}\cap\O)\, d s\label{e141}
\end{align}
Then, thanks to the Observation \ref{o1} and Remark \ref{r2.1.5}, we have
\begin{align}
\bigg(\int_{\O}&u_h^{p^*} d x\bigg)^\frac{1}{a}+\bigg(\int_{\p\O}u_h^{p_*} d \s\bigg)^\frac{1}{b}\le C_1 \int_{\O} |\n u_h|^p d x\non\\
&\le C_1\bigg(\int_{\p\O}u_h^{p_*} d \s+h\int_h^\infty s^{q-1}\H^{n-1}(\{u>s\}\cap\p\O)\, d s\non\\
&+d\int_{\O}u_h^{p^*} d x+d\, h\int_h^\infty s^{p^*-2}\H^{n}(\{u>s\}\cap\O)\, d s\bigg)\label{e142}
\end{align}
where $C_1$ always depend only on $n,p$.\\
Since $|u_h|\le |u|,\,u\in L^{p_*}(\p\O)\cap L^{p^*}(\O),\,u_h\ra 0$ as $h\ra 0$, by dominated convergence theorem, we get
\begin{equation}
\int_{\O}u_h^{p^*} d x\ra0,\,
\int_{\p\O}u_h^{p_*} d \s\ra0\,\,\,\T{as}\,\,\,h\ra 0.\label{e143}
\end{equation}
It follows from $(\ref{e142})$ and $(\ref{e143})$ that for small $h>0$, we have
\begin{align}
&\bigg(\int_{\O}u_h^{p^*} d x\bigg)^\frac{1}{a}+\bigg(\int_{\p\O}u_h^{p_*} d \s\bigg)^\frac{1}{b}\non\\
&\le C_2 h\bigg(\int_h^\infty s^{q-1}\H^{n-1}(\{u>s\}\cap\p\O)\, d s+ \int_h^\infty s^{p^*-2}\H^{n}(\{u>s\}\cap\O)\, d s\bigg)\label{e144}
\end{align}
where $C_2$ always depend only on $n,p,d$.\\
We now define 
\begin{align*}
   & G(h):=\int_h^\infty s^{q-1}\H^{n-1}(\{u>s\}\cap\p\O)\, d s,\\
    &F(h):=\int_h^\infty s^{p^*-2}\H^{n}(\{u>s\}\cap\O)\, d s.
\end{align*}
Then we have $G'(h)=-h^{q-1}\,\H^{n-1}(W_h\cap\p\O)\le 0,F'(h)=-h^{p^*-2}\,\H^{n}(W_h\cap\O)\le 0$ when $h>0$, which implies $G(0):=\D\operatorname*{lim}_{h\ra 0^+}G(h)$ and $F(0):=\D\operatorname*{lim}_{h\ra 0^+}F(h)$ exists (but may be infinity). In the following, we will show that $G(0)+F(0)<\infty$. \\
It follows from $(\ref{e138})$ and $(\ref{e139})$ that as $h\ra 0$, we have
\begin{align*}
h^{p^*}\,\H^{n}(W_h\cap\O)\le \int_{\O}u_h^{p^*} d x\ra 0,\,\,
h^{p_*}\,\H^{n-1}(W_h\cap\p\O)\le\int_{\p\O}u_h^{p_*} d \s\ra0
\end{align*}
Then from $(\ref{e144})$, for small $h>0$, we have
\begin{align}
&\bigg(h^{p^*}\,\H^{n}(W_h\cap\O)\bigg)^\frac{1}{b}+\bigg(h^{p_*}\,\H^{n-1}(W_h\cap\p\O)\bigg)^\frac{1}{b}\non\\
&\le\bigg(h^{p^*}\,\H^{n}(W_h\cap\O)\bigg)^\frac{1}{a}+\bigg(h^{p_*}\,\H^{n-1}(W_h\cap\p\O)\bigg)^\frac{1}{b}\non\\
&\le C_2 h\bigg(\int_h^\infty s^{q-1}\H^{n-1}(\{u>s\}\cap\p\O)\, d s+ \int_h^\infty s^{p^*-2}\H^{n}(\{u>s\}\cap\O)\, d s\bigg)\label{e145}
\end{align}
That is, for small $h>0$, we have
\begin{align}
    &(-F'(h)h^2)^\frac{1}{b}+(-G'(h)h^2)^\frac{1}{b}\le C_2 h(G(h)+F(h))\non\\
    \RA &-(F+G)'(h)\le C_2 h^{b-2}(G(h)+F(h))^b\non\\
    \RA &((F(h)+G(h))^{1-b})'\le C_2 h^{b-2}\label{e146}
\end{align}
By integrating $(\ref{e146})$, we obtain
\begin{align}\label{e147}
    (F(h)+G(h))^{1-b}-(F(0)+G(0))^{1-b}\le C_2 h^{b-1}\quad\T{for small}\,\,h>0.
\end{align}
Arguing as Lemma \ref{l3}, we have 
\begin{align}
    h G(h)\ra 0,\,h F(h)\ra 0\,\,\T{as}\,\,h\ra 0.\label{e148}
\end{align}
It follows from $(\ref{e147})$ and $(\ref{e148})$ that $G(0),F(0)<\infty$. By using $(\ref{e145})$ and $G,F$ is non-increasing, for small $h>0$ we get
\begin{align}
\bigg(h^{(p-1)a}\,\H^{n}(W_h\cap\O)\bigg)^\frac{1}{a}+\bigg(h^{(p-1)b}\,\H^{n-1}(W_h\cap\p\O)\bigg)^\frac{1}{b}\le C_1 (G(0)+F(0)).\non
\end{align}
Finally, arguing as Lemma \ref{l3}, we finish the proof.

\end{proof}

\begin{lem}\label{l8}
	Suppose $u\in W^{1,p}\cap L^\infty(U\cap \O )$ is a positive solution to
    \begin{equation}\label{e149}
    \left\{
    \begin{array}{lr}
	\Delta_p u=-g(x) u^{p-1} & ,\T{in} \,\,\O\cap U\\
	|D u|^{p-2}\frac{\partial u}{\partial x_n}=-f(x) u^{p-1} & ,\T{on} \,\,\p\O\cap U
	\end{array}
 \right.     
\end{equation}
where $|g(x)|\le \L$ in $U\cap\O$, $|f(x)|\le \l$ in $U\cap\p\O$ and $\L,\l\ge 0$.\\
Then there exist $\d,p_1>0$ with the following property: let $c_n\in(0,1),\l$ be such that 
\begin{align*}
    \T Vol (B_{c_n})\le 1,\,\,\T Area (B_{c_n}\cap \p\O)\le 1,\,\,\l\le\d,
\end{align*}
   then for any $0<\rho\le c_n$, $x_0\in\O\cap W$, we have 
   $$\inf_{B_{\rho/2}(x_0)\cap \O} u\ge c\|u\|_{L^{p_1}(B_{\rho}(x_0)\cap \O)},$$
   where $\d,\,p_1$ depends only on $n,p$; and $c$ depends only on $n,p,\rho,\L$.
\end{lem}
\begin{proof}
The proof of Lemma \ref{l8} is similar to the one of Lemma \ref{l4} and we only point out hereafter the changes required for the solution to $(\ref{e149})$.\\
We argue as Lemma \ref{l4}.
Set $w=u^{-1}$. From $(\ref{e149})$ we have
    \begin{equation}
   \int_\O |\n u|^{p-2}\n u\cdot\n \vp d x=\int_{\p\O}f(x)u^{p-1} \vp d\s+\int_{\O}g(x)u^{p-1} \vp d x,\,\,\,\forall \vp\in C_c^\infty(U).
    \label{e150}  
   \end{equation}
\textbf{Step 1:}\textit{ we prove that $\D\inf_{B_{\rho/2}(x_0)\cap \O} u\ge c\bigg(\|u^{-1}\|_{L^t(B_{\rho}(x_0)\cap \O)}\bigg)^{-1}$ for any $t>0$}.\\
Let $\phi\in C_c^\infty(U)$ and use $\vp=w^{2(p-1)}\phi$ as a test function in $(\ref{e150})$, we obtain
\begin{align}
   \int_\O |\n w|^{p-2}\n w\cdot\n \phi d x=&-\int_{\p\O}f(x)w^{p-1} \phi d\s-2(p-1)\int_\O |\n w|^{p} w^{-1} \phi d x\non\\
   &-\int_{\O}g(x)w^{p-1} \phi d x\non\\
      \le& -\int_{\p\O}f(x)w^{p-1} \phi d\s-\int_{\O}g(x)w^{p-1} \phi d x.
   \label{e151}
\end{align}
Then $(\ref{e59})$ becomes
\begin{align}
     \int_\O |\n(\eta  w^{\frac{\b+p-1}{p}})|^p d x&\le C_1\bigg(\int_\O |\n \eta|^p w^{\b+p-1} dx+\l\b^p\int_{\p\O} \eta^p  w^{\b+p-1} d\s\non\\
     &+\L\b^p\int_\O |\n \eta|^p w^{\b+p-1} dx\bigg),
     \label{e152}
\end{align}
where $C_1$ always depend only on $p$.\\
Let $C_2$ always depends only on $n,p,\L$; and $C_3$ always depends only on $n,p,\rho,\L$.
Then we argue as Lemma \ref{l4} and finish the proof of step 1.\\
Set $$v=\log u-\frac{\int_{\O\cap B_{\rho}(x_0)}\log u \,d x}{|\O\cap B_{\rho}(x_0)|}.$$
\textbf{Step 2:}\textit{ we prove that $v\in BMO$}.\\
Using $\vp=u^{-(p-1)}\eta^p$ as a test function in $(\ref{e150})$, then $(\ref{e61})$ becomes
\begin{align}
   (p-1)\int_\O |\n v|^{p} \eta^p d x=&-\int_{\p\O}f(x) \eta^p d\s-\int_{\O}g(x) \eta^p d x\non\\
   &-p\int_\O |\n v|^{p-2}\n v\cdot\n \eta \eta^{p-1} d x\non\\
     \RA \int_\O |\n v|^{p} \eta^p d x\le& C_1\bigg(\l\int_{\p\O}\eta^p d\s+\int_\O |\n \eta|^{p}d x+\L\int_{\O}\eta^p d x\bigg).
   \label{e153}
\end{align}
For any ball $B_{r}(y)\subset B_{\rho}(x_0)$, we choose $\eta$ such that $\eta=1$ in $B_{r}(y)$, $\eta=0$ outside $B_{2 r}(y)$ and $|\n \eta|\le C_n/r$. Then by Poincare inequality and $(\ref{e153})$, we have
\begin{align}
   \int_{\O\cap B_r(y)} | v-(v)_r|^{p} d x&\le C_2 r^p\int_{\O\cap B_r(y)} |\n v|^{p} d x\non\\
   &\le C_2 r^p(\l r^{n-1}+r^{n-p}+\L r^{n})\non\\
   &\le C_2 r^n,\non
\end{align}
where $(v)_r:=\frac{\int_{\O\cap B_{r}(y)}v d x}{|\O\cap B_{r}(y)|}$.\\
Then we argue as Lemma \ref{l4} and finish the proof of step 2.\\
\textbf{Step 3:}\textit{ for any $p_1>0$, we have
   $\inf_{B_{\rho/2}(x_0)\cap \O} u\ge c\|u\|_{L^{p_1}(B_{\rho}(x_0)\cap \O)}.$}\\
The proof of step 3 is the same as Lemma \ref{l4}.

\end{proof}
Using Lemma \ref{l7}, Lemma \ref{l8} instead of Lemma \ref{l3}, Lemma \ref{l4} and arguing as Corollary \ref{cor1}, we obtain
\begin{cor}\label{cor2}
    Suppose $u$ is a solution of $(\ref{e124})$. Then there exists $R_3>1$ such that
   $$\sup_{W\cap \O} u_R\le C\inf_{W\cap \O} u_R\quad\T{for}\,\,R\ge R_3,$$
   where $C$ depends only on $n,p$.
\end{cor}
Therefore, using Lemma \ref{l7}, Lemma \ref{l8}, Corollary \ref{cor2} instead of Lemma \ref{l3}, Lemma \ref{l4}, Corollary \ref{cor1} and arguing as Proposition \ref{p1}, we obtain
\begin{prop}\label{p4}
   Let $n\ge 2,\,1<p<n,$ and let $u$ be a solution to $(\ref{e124})$. Then there exist  $C_0,\,C_1,\,\a>0$ such that 
   \begin{align*}
       u(x)\le &\frac{C_1}{1+|x|^\frac{n-p}{p-1}}\,,\,\,|\n u(x)|\le \frac{C_1}{1+|x|^\frac{n-1}{p-1}}\,\,\,\T{in} \,\,\,\O\\
       &u(x)\ge\frac{C_0}{1+|x|^\frac{n-p}{p-1}}\quad\T{in} \quad\O\\
       \big||x|^\frac{n-p}{p-1}u(x)&- \a\big|+\big||x|^\frac{n-1}{p-1}\n u(x)+\a\frac{n-p}{p-1}x\big|\ra 0\,\,\T{as}\,\,|x|\ra\infty
   \end{align*}
\end{prop}
Also, adapting the proof of Proposition \ref{p2}, we can get the following proposition.
\begin{prop}\label{p5}
   Let $u$ be a solution to $(\ref{e124})$. Then $|\n u|^{p-2}\n u\in W^{1,2}_{loc}(\ov\O)$, and for any
$\g\in\rz$ the following asymptotic estimate holds:
\begin{align}
    &\int_{B_R\cap\O}|\n(|\n u|^{p-2}\n u)|^2 u^\g d x\le C(1+R^{-n-\g\frac{n-p}{p-1}})\qquad \forall\, R>1,\non\\
    &\int_{\O\cap(B_{2 R}\bs B_R)}|\n(|\n u|^{p-2}\n u)|^2 u^\g d x\le C R^{-n-\g\frac{n-p}{p-1}}\qquad \forall\, R>1,\non
\end{align}
   where $C$ is a positive constant independent of $R$.
\end{prop}
\begin{proof}
The proof of Proposition \ref{p5} is similar to the one of Proposition \ref{p2} and we only point out hereafter the changes required for the solution to $(\ref{e124})$.\\
We argue as Proposition \ref{p2}.
For $k$ fixed, we let $u_k\in D^{1,p}(\O)$ be a solution of 
\begin{equation}\label{e154}
    \left\{
    \begin{array}{lr}
	\div (a^k(\n u_k))=-d u^{p^*-1} & \T{in} \,\,\O,\\
	a^k(\n u_k)\cdot \nu=\pm u^q & \T{on}\,\, \p \O,
	\end{array}
 \right.
\end{equation}
where $\nu$ is the unit outward normal of $\p\O$.\\
Then $(\ref{e68})$ becomes
\begin{align}
    \min_{v\in C_0^\infty(B_R)}\bigg\{\int_{\O\cap B_R}\frac{1}{p}H_k(\n v)\, d x-d\int_{\O\cap B_R}u^{p^*-1} v\, d x\mp\int_{\p\O\cap B_R}u^q v\, d\s\bigg\}.\label{e155}
\end{align}
And $(\ref{e69})$ becomes
\begin{equation}\label{e156}
    \left\{
    \begin{array}{lr}
	\div (a^k(\n u_k^R))=-d u^{p^*-1} & \T{in} \,\,\O\cap B_R,\\
	a^k(\n u_k^R)\cdot \nu=\pm u^q & \T{on}\,\, \p\O\cap B_R,\\
 u_k^R=0 & \T{on}\,\, \O\cap \p B_R.
	\end{array}
 \right.
\end{equation}
Similar to $(\ref{e70})$, we have
\begin{align}
  \|u_k^R\|_{L^{p_*}(\p\O)}\le C,\,\|u_k^R\|_{L^{p^*}(\O)}\le C,\,\|\n u_k^R\|_{L^{p}(\O)}\le C.
    \label{e157}
\end{align}
Then $(\ref{e72})$ becomes
\begin{equation}\label{e158}
    \left\{
    \begin{array}{lr}
	\div (a(\n \ov u))=-d u^{p^*-1} & \T{in} \,\,\O,\\
	a(\n \ov u)\cdot \nu=\pm u^q & \T{on}\,\, \p \O.
	\end{array}
 \right.
\end{equation}
Arguing as Proposition \ref{p2}, we see $u=\ov u$. To simplify the notation, we shall drop the dependency of $a$ on k and we write $a$ instead of $a^k$ respectively.\\
Let $\psi\in C_c^{0,1}(B_{4R})$ and test $(\ref{e154})$ with $\psi$, we have
\begin{align}
   \int_{\O} a(\n u_k)\cdot\n \psi \,d x=
    \pm\int_{\p\O} u^q \psi\,d\s+d\int_{\O} u^{p^*-1} \psi\,d x.
    \label{e159}
\end{align}
Let $\vp\in C^2( U)\cap C^{1}(\ov{ U})$ and $\operatorname{supp}(\vp)\Subset B_{4R}$. Using $\psi=\p_m\vp\,\xi_\d$ in $(\ref{e159})$ with $1\le m\le n$ and integrating by parts, $(\ref{e74})$ becomes
\begin{align}
    0=&\sum_i\bigg(-\int_{ U}\p_{i}\vp\,\p_m(a_i(\n u_k)\xi_\d) d x+\int_{ U}a_i(\n u_k)\p_{m}\vp\,\p_i \xi_\d d x\bigg)\non\\
    &+d\int_U \vp \,\p_m(u^{p^*-1}\xi_\d) d x
    \non\\
    \RA &\sum_i\int_{ U}\p_{i}\vp\p_m a_i(\n u_k)\xi_\d d x=\sum_i\bigg(-\int_{ U}\p_{i}\vp\,a_i(\n u_k)\p_m\xi_\d d x\non\\
    &+\int_{ U}a_i(\n u_k)\p_{m}\vp\,\p_i \xi_\d d x\bigg)
    +d\int_U \vp \,\p_m(u^{p^*-1}\xi_\d) d x
    .\label{e160}
\end{align}
Passing to the limit as $\d\ra 0$ into $(\ref{e160})$, we get
\begin{align}
    &\sum_i\int_{ U}\p_{i}\vp\,\p_m a_i(\n u_k) d x=\sum_i\bigg(\int_{ \p \O}\p_{i}\vp\,a_i(\n u_k)\nu^m d \s\non\\
    &-\int_{\p\O}a_i(\n u_k)\p_{m}\vp\,\nu^i d \s\bigg)+d\int_U \vp \,\p_m(u^{p^*-1}) d x-d\int_{\p\O} \vp \,u^{p^*-1}\nu^m d \s.\label{e161}
\end{align}
Let $\eta\in C_c^\infty(B_{4R})$. We choose $\vp=a_m(\n u_k) u^\g \eta^2$ and divide $m$ into two cases.\\
\textbf{Case 1:} $m\neq n$\\
Then, it follows from $(\ref{e154})$ that $(\ref{e161})$ becomes
\begin{align}
    &\sum_i\int_{U}\p_m a_i(\n u_k)\,\p_{i}(a_m(\n u_k) u^\g \eta^2)\, d x\non\\
    =&\mp\sum_i\int_{\p\O}u^q\p_{m}\vp\, d \s
    +d\int_U \vp \,\p_m(u^{p^*-1}) d x\non\\
=&\pm\sum_i\int_{\p\O}\p_m(u^q)a_m(\n u_k) u^\g \eta^2\, d \s+d\int_U \p_m(u^{p^*-1})a_m(\n u_k) u^\g \eta^2\, d x
    \label{e162}
\end{align}

\noindent\textbf{Case 2:} $m=n$\\
Note that $\vp=\mp u^{q+\g}\eta^2$ on $\p\O$ in this case.
Then, $(\ref{e161})$ becomes
\begin{align}
    &\sum_i\int_{ U}\p_{i}\vp\,\p_m a_i(\n u_k) d x\non\\
    =&\sum_{i=1}^n\bigg(\int_{ \p \O}\p_{i}\vp\,a_i(\n u_k)\nu^n d \s\bigg)-\int_{\p\O}a_n(\n u_k)\p_{n}\vp\,\nu^n d \s\non\\
    &+d\int_U \vp \,\p_n(u^{p^*-1}) d x+d\int_{\p\O} \vp \,u^{p^*-1} d \s\non\\
    =&\pm\sum_{i=1}^{n-1}\int_{ \p \O}\p_{i}(u^{q+\g}\eta^2)\,a_i(\n u_k) d \s+d\int_U \vp \,\p_n(u^{p^*-1}) d x\non\\
&+d\int_{\p\O}\,u^{p^*+p_*+\g-2}\eta^2 d \s
    \label{e163}
\end{align}
Combining $(\ref{e162})$ and $(\ref{e163})$, we obtain
\begin{align}
&\sum_{i,m=1}^n\bigg|\int_{U}\p_m a_i(\n u_k)\,\p_{i}(a_m(\n u_k) u^\g \eta^2)\, d x\bigg|\non\\
&\le C\bigg(\int_{\p\O}|a(\n u_k)|\,|\n u| u^{\g+q-1} \eta^2\, d \s+\int_{\p\O}|a(\n u_k)|\,|\n \eta| u^{\g+q} \eta\, d \s\non\\
&+\int_{\p\O}\,u^{p^*+p_*+\g-2}\eta^2 d \s
+\int_{\O}|a(\n u_k)|\,|\n u| u^{\g+p^*-2} \eta^2\, d x
   \bigg) \label{e164}
\end{align}
where $C$ only depends on $n,p,\g,d$.\\
Finally, arguing as Proposition \ref{p2}, we finish the proof.

\end{proof}

Adapting the proof of Proposition \ref{p3}, we can get the following proposition.
\begin{prop}\label{p6}
    Let $u$ be a solution to $(\ref{e109})$. Then we have
     $$\p_j (|\n u|^{p-2}u_i)=q(\frac{|\n u|^{p-2}u_i u_j}{u}-\frac{|\n u|^p}{n u}\d_{i j})-\frac{d}{n}u^{p^*-1}\d_{ij}\quad\T{in}\,\,\,L^2_{loc}(\ov \O).$$
     In particular, we have $|\n u|^{p-2}\n u\in C^{1,\a}_{loc}(\ov\O)$.
\end{prop}
\begin{proof}
Arguing as the proof of Proposition \ref{p3}, we obtain $u_k\in D^{1,p}(\O)$ being a solution of 
\begin{equation}\label{e165}
    \left\{
    \begin{array}{lr}
	\div (a^k(\n u_k))=-d u^{p^*-1} & \T{in} \,\,\O,\\
	a^k(\n u_k)\cdot \nu=\pm u^q & \T{on}\,\, \p \O.
	\end{array}
 \right.
\end{equation}

Also, we have $u_k\in C^{3,\a}(\O)\cap C^{2,\a}_{loc}(\ov\O)$ and $u_k\in C^{1,\a}_{loc}(\ov \O)$, uniformly in $k$. Hence, $(u_k)_{k\in\nz}$
converges up to a subsequence in
$C^1_{loc}(\ov \O)$ to $u$. Besides, the proof of Proposition \ref{p5} shows that as $k\ra\infty$, up to a subsequence, we have $a^k(\n u_k)\ru a(\n u)$ in 
$W^{1,2}_{loc}(\ov\O)$. In the following, we only consider the subsequence of $(u_k)_{k\in\nz}$ satisfying that $u_k\ra u$ in $C^1_{loc}(\ov \O)$ and $a^k(\n u_k)\ru a(\n u)$ in 
$W^{1,2}_{loc}(\ov\O)$.\\
Set $a,X,L,a^k,X^k,L^k,H_k,D^k$ as before and denote
\begin{align}
   E_{ij}=X_{ij}+\frac{d}{n}u^{p^*-1}\d_{ij},\,E^k_{ij}=X^k_{ij}+\frac{d}{n}u^{p^*-1}\d_{ij}.
    \label{e166}
\end{align}
Arguing as Proposition \ref{p3}, we obtain
\begin{equation}\label{e167}
    \D\sum_{i,j}(E^k_{ij}-q L^k_{ij})(E^k_{ji}-q L^k_{ji})\ge C_p\sum_{i,j}|E^k_{ij}-q L^k_{ij}|^2 \ge 0.
\end{equation}
To simplify the notation, we shall also drop the dependency on k and we write $a$, $\tl u$, $X$, $X_i$, $X_{ij}$, $E_{ij}$, $L_{ij}$, $D_{ij}$ instead of $a^k,\,u_k,\,X^k,\,X^k_i,\,X^k_{ij},\,E^k_{ij},\,L^k_{ij}$, $D^k_{ij}$ respectively.\\
First, using $(\ref{e165})$, that is $\sum_i X_{ii}=-d u^{p^*-1}$, we get the following Serrin-Zou type differential equality
\begin{align}
    &\sum_i\p_i\bigg(\sum_j u^{1-q} X_{ij}X_j-\frac{(n-1)q}{n} u^{-q}|X|^{\frac{p}{p-1}}X_i+\frac{d}{n}u^{p^*-q}X_i\bigg)\non\\
    &=\sum_{i,j}\bigg(u^{1-q} X_{iij}X_j+u^{1-q} X_{ij}X_{ji}
    +(1-q)u^{-q} X_{ij}X_j u_i\non\\
    &-\frac{(n-1) p}{n-p}u^{-q}|X|^{\frac{1}{p-1}}\frac{X_j X_{ji}X_i}{|X|}\bigg)+\frac{(n-1)q^2}{n} \sum_i u^{-q-1}|X|^{\frac{p}{p-1}}X_i u_i\non\\
    &+\sum_i\bigg(
    \frac{d}{n}u^{p^*-q}X_{ii}+\frac{d(p^*-q)}{n}u^{p^*-p_*}X_{i}u_i-\frac{(n-1)q}{n}u^{-q}|X|^{\frac{p}{p-1}}X_{ii}
    \bigg)\non\\
    &=\sum_{i,j}\bigg(u^{1-q} X_{ij}X_{ji}
    -2 q u^{-q} X_{ij}X_j \tl{u}_i\bigg)+\frac{(n-1)q^2}{n}\sum_i u^{-q-1}|X|^{\frac{p}{p-1}}X_i u_i\non\\
    &+\sum_i\bigg(\frac{d(n-1)q}{n}u^{p^*-p_*}|X|^{\frac{p}{p-1}}
    -\frac{d}{n}u^{2 p^*-p_*}-\frac{d(n+1)(p-1)}{n-p}u^{p^*-p_*}X_{i}u_i
    \bigg)\non\\
    &+(1-q)\sum_{i,j} u^{-q} X_{ij}X_j (u_i-\tl{u}_i)+(q+1)P^k\non\\
    &=\sum_{i,j}u^{1-q} (X_{ij}-q L_{ij})(X_{ji}-q L_{ji})-\frac{d}{n}u^{2 p^*-p_*}+E^k+M^k\non\\
    &=\sum_{i,j}u^{1-q} (E_{ij}-q L_{ij})(E_{ji}-q L_{ji})+E^k+M^k,\label{e168}
\end{align}
where 
\begin{align}
    M^k=&-\frac{2 q}{n}\sum_i u^{-q}(\frac{1}{k^2}+|\n\tl u|^2)^\frac{p}{2}X_{ii}
    -2 q\sum_{i,j} u^{-q}X_{ij}D_{ij}\non\\
    &+\frac{d(n-1) q}{n}u^{p^*-p_*}|X|^{\frac{p}{p-1}}
    -\frac{d(n+1)(p-1)}{n-p}u^{p^*-p_*}X_{i}u_i\non\\
    =&-2 q\sum_{i,j} u^{-q}X_{ij}D_{ij}
    +\frac{d(n-1) q}{n}u^{p^*-p_*}\big(|X|^{\frac{p}{p-1}}-(\frac{1}{k^2}+|\n\tl u|^2)^\frac{p}{2}\big)\non\\
    &+\frac{d(n+1)(p-1)}{n-p}u^{p^*-p_*}(\frac{1}{k^2}+|\n\tl u|^2)^\frac{p-2}{2}\frac{1}{k^2},
    \label{e169}
\end{align}
and $E^k$ is as (\ref{e91}).

\begin{comment}
\begin{align}
    E^k=&\sum_{i,j}\bigg((1-q) u^{-q} X_{ij}X_j (\tl{u}_i-u_i)-2 q u^{1-q} X_{ij}D_{ij}
    \bigg)\non\\
    &+\frac{(n-1)q^2}{n}\sum_i u^{-q-1}|X|^{\frac{p}{p-1}}X_i (u_i-\tl{u}_i)+(q+1)P^k\non\\
    &+\frac{(n-1)q^2}{n}\sum_i u^{-q-1}|X|^{\frac{p}{p-1}}X_i \tl{u}_i-q^2\sum_{i,j}u^{1-q} L_{ij}L_{ji}.\non
    \end{align}
\end{comment}

    Note that we have shown in Proposition \ref{p3} that $E^k\ra 0$ in $L^2_{loc}(\ov\O)$ and $D^k_{ij}\ra 0$ in $C^0_{loc}(\ov\O)$.\\
Besides, since $X^k\in W^{1,2}_{loc}(\ov \O)$, uniformly in k, we obtain that $u^{-q}X^k_{ij}D^k_{ij}\ra 0$ in $L^2_{loc}(\ov\O)$, and then $M^k\ra 0$ in $L^2_{loc}(\ov\O)$.

Second, using $(\ref{e165})$, we get the following Pohozaev-type differential equality
\begin{align}
    &\sum_i\p_n(X_i \tl{u}_i)-p\sum_i\p_i(X_i \tl{u}_n)-d\frac{p}{p^*}\p_n(u^{p^*})\non\\
    =&\sum_i X_{i n}\tl{u}_i-(p-1)\sum_i X_i \tl{u}_{n i}-p\sum_iX_{ii} \tl{u}_n-d p u^{p^*-1}u_n\non\\
    =&\sum_i X_{i n}\tl{u}_i-(p-1)\sum_i X_i \tl{u}_{n i}+d p u^{p^*-1}(\tl{u}_n- u_n)\non\\
    =&I^k+d p u^{p^*-1}(\tl{u}_n- u_n):=N^k\label{e170}
\end{align}
where
$$I^k=-\bigg(1+\frac{(p-2)|\n\tl{u}|^2}{\frac{1}{k^2}+|\n \tl{u}|^2}\bigg)^{-1}\frac{(p-2)\frac{1}{k^2}}{\frac{1}{k^2}+|\n \tl{u}|^2}\sum_i X_{i n}\tl{u}_i.$$
Recall that $I^k\ra 0$ in $L^2_{loc}(\ov\O)$ and $u_k\ra u$ in $C^1_{loc}(\ov \O)$, so we have $N^k\ra 0$ in $L^2_{loc}(\ov\O)$.\\
Then, it follows from $(\ref{e168})$ and $(\ref{e170})$ that
\begin{align}
    &\sum_i\p_i\bigg(\sum_j u^{1-q} X_{ij}X_j-\frac{(n-1)q}{n} u^{-q}|X|^{\frac{p}{p-1}}X_i+\frac{d}{n}u^{p^*-q}X_i\bigg)\non\\
    &\pm\frac{n-1}{n-p}\bigg[\sum_i\p_n(X_i \tl{u}_i)-p\sum_i\p_i(X_i \tl{u}_n)-d\frac{p}{p^*}\p_n(u^{p^*})\bigg]\non\\
    &=\sum_{i,j}u^{1-q} (E_{ij}-q L_{ij})(E_{ji}-q L_{ji})+E^k+M^k\pm\frac{n-1}{n-p}N^k
    \label{e171}
\end{align}
Let $\eta\in C_c^\infty(B_R)$. We  multiply $(\ref{e171})$ by $\eta$ and integrate over $\O$. Then, it follows from the divergence theorem that
\begin{align}
    &\int_\O \bigg(
    \sum_{i,j}u^{1-q} (E_{ij}-q L_{ij})(E_{ji}-q L_{ji})+E^k+M^k\pm\frac{n-1}{n-p}N^k
    \bigg)\eta\, d x\non\\
    &=-\int_\O\sum_i\big(\sum_j u^{1-q} X_{ij}X_j-\frac{(n-1)q}{n} u^{-q}|X|^{\frac{p}{p-1}}X_i+\frac{d}{n}u^{p^*-q}X_i
    \big)\eta_i d x\non\\
    &\mp\frac{n-1}{n-p}\int_\O\big(\sum_i X_i \tl{u}_i\eta_n-p\sum_i X_i \tl{u}_n\eta_i-d\frac{p}{p^*}u^{p^*}\eta_n
    \big) d x\non\\
    &-\int_{\p\O}\big(\sum_j u^{1-q} X_{nj}X_j-\frac{(n-1)q}{n} u^{-q}|X|^{\frac{p}{p-1}}X_n+\frac{d}{n}u^{p^*-q}X_n
    \big)\eta d\s\non\\
    &\mp\frac{n-1}{n-p}\int_{\p\O}\big(\sum_i X_i \tl{u}_i-p X_n \tl{u}_n-d\frac{p}{p^*}u^{p^*}
    \big)\eta d\s\label{e172}
\end{align}
Since $X^k_{n}=\mp u^q$ on $\p\O$ and $\sum_i X^k_{ii}=-d u^{p^*-1}$, we have
\begin{align}
    &\int_{\p\O}\big(\sum_j u^{1-q} X_{nj}X_j-\frac{(n-1)q}{n} u^{-q}|X|^{\frac{p}{p-1}}X_n+\frac{d}{n}u^{p^*-q}X_n
    \big)\eta d\s\non\\
=&\int_{\p\O}\pm\big(\frac{(n-1)q}{n} |X|^{\frac{p}{p-1}}-q\sum_{j=1}^{n-1} u_j X_j-u^{1-q}\sum_{j=1}^{n-1}X_{jj}X_n+\frac{(n-1)d}{n}u^{p^*}
    \big)\eta d\s\non\\
    %=&\int_{\p\O}\pm\big(\frac{(n-1)q}{n} |X|^{\frac{p}{p-1}}-q\sum_{j=1}^{n-1} u_j X_j-\sum_{j=1}^{n-1}X_{j}u_j+\frac{(n-1)d}{n}u^{p^*} \big)\eta d\s\non\\
    %\mp&\int_{\p\O}\sum_{j=1}^{n-1}u X_{j}\eta_j d\s\non\\
    =&\int_{\p\O}\pm\big(-\frac{(n-1)p}{n-p}\sum_{j=1}^{n-1} X_j\tl{u}_j +\frac{(n-1)q}{n} \sum_{j=1}^{n} X_j\tl{u}_j+\frac{(n-1)d}{n}u^{p^*}+J^k\big)\eta d\s\non\\
    \mp&\int_{\p\O}\sum_{j=1}^{n-1}u X_{j}\eta_j d\s\non\\
    =&\mp\frac{n-1}{n-p}\int_{\p\O}\big(\sum_i X_i \tl{u}_i-p X_n \tl{u}_n-d\frac{p}{p^*}u^{p^*}
    \big)\eta d\s\pm\int_{\p\O}J^k\eta\mp\sum_{j=1}^{n-1}u X_{j}\eta_j d\s\label{e173}
\end{align}
Finally, arguing as Proposition \ref{p3}, we finish the proof.

\end{proof}
Also, we consider the auxiliary function
$
    v=\frac{n-p}{p}u^{-\frac{p}{n-p}},\non
$
where $u$ is a solution of $(\ref{e109})$. A straightforward computation shows that $v > 0$ satisfies the
following equation
\begin{equation}
\left\{
    \begin{array}{lr}
	\Delta_p v=\frac{n(p-1)}{p}\frac{|\n v|^p}{v}+d\frac{n-p}{p}\frac{1}{v} & \T{in}\quad \O\\
	|\n v|^{p-2}\frac{\partial v}{\partial t}=\pm 1 & \T{on}\quad \p\O
	\end{array}
 \right.     
\end{equation}
In particular, $v\in C_{loc}^{1,\a}(\ov\O)$ for some $\a\in(0,1)$.

Using Proposition \ref{p6} instead of Proposition \ref{p3} and arguing as Theorem \ref{t1}, we obtain Theorem \ref{t2}.

\end{document}